\documentclass[a4paper]{amsart}

\usepackage{graphicx}

\newtheorem{theorem}{Theorem}[section]
\newtheorem{lemma}[theorem]{Lemma}

\newtheorem{proposition}[theorem]{Proposition}
\newtheorem{corollary}[theorem]{Corollary}

\begin{document}

\markboth{Gyo Taek Jin and Hwa Jeong Lee}
{Prime knots whose arc index is smaller than the crossing number}

\title{Prime knots whose arc index is smaller \\ than the crossing number}

\author{Gyo Taek Jin and Hwa Jeong Lee}

\address{Department of Mathematical Sciences, KAIST, Daejeon 305-701 Korea}
\email{trefoil@kaist.ac.kr, hjwith@kaist.ac.kr}

\begin{abstract}
 It is known that the arc index of alternating knots is the minimal crossing number plus two and the arc index of prime nonalternating knots is less than or equal to the minimal crossing number. We study some cases when the arc index is strictly less than the minimal crossing number. We also give minimal grid diagrams of some prime nonalternating knots with 13 crossings and 14 crossings whose arc index is the minimal crossing number minus one.
\end{abstract}

\keywords{knot, link, arc presentation, arc index, grid diagram, minimal crossing number, Kauffman polynomial, filtered tree}

\subjclass[2000]{57M25, 57M27}
\maketitle

\section{Arc presentation}

A link\footnote{It is a knot if it has only one component.} can be embedded in a book of finitely many half planes in $\mathbb R^3$ so that each half plane intersects the link in a single arc. Such an embedding is called an \emph{arc presentation\/} of the link. The minimal number of half planes among all arc presentations of a link is
called the \emph{arc index\/} of the link. The arc index of a link $L$ is denoted by $\alpha(L)$.

Suppose we have an arc presentation of a link $L$. In each half plane containing a single arc of $L$, we deform the arc into the union of two horizontal arcs and one vertical arc with the two end points fixed. Then we have a new arc presentation of $K$ which looks like the figure in the left of Figure~\ref{fig:trefoil}. Relaxing the pairs of consecutive horizontal arcs off the axis, we obtain a diagram of $L$ as shown in the right of Figure~\ref{fig:trefoil}. The new diagram is called a \emph{grid diagram\/}. A grid diagram is a link diagram which is the union of a finitely many vertical strings and the same number of horizontal strings with the property that at every crossing the vertical string crosses over the horizontal string. The minimal number of vertical strings among all grid diagram of a link is equal to the arc index of the link.

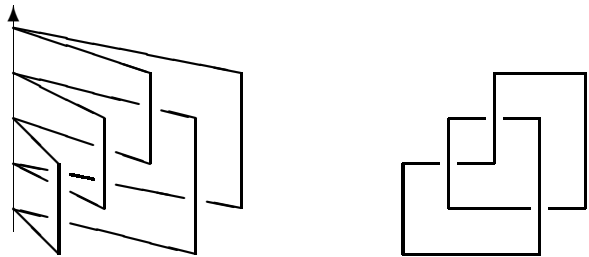
\begin{figure}[h]
\setlength{\unitlength}{0.6mm}
\centering
\begin{picture}(50,55)(-10,0)
\thicklines
\put(-10,10){\line(1,-1){10}}
\put(-10,10){\line(4,-1){7.5}} \put(2.5,6.875){\line(4,-1){27.5}}
\put(-10,20){\line(2,-1){7.5}} \put(2.5,13.75){\line(2,-1){7.5}}
\put(-10,20){\line(5,-1){7.5}} 
      \qbezier(2.5,17.6)(5,17.1)(7.5,16.6)
      \put(12.5,15.4){\line(5,-1){15}} \put(32.5,11.6){\line(5,-1){7.5}}
\put(-10,30){\line(1,-1){10}}
\put(-10,30){\line(3,-1){17.5}} \put(12.5,22.5){\line(3,-1){7.5}}
\put(-10,40){\line(2,-1){20}}
\put(-10,40){\line(4,-1){27.5}} \put(22.5,31.875){\line(4,-1){7.5}}
\put(-10,50){\line(5,-1){50}}
\put(-10,50){\line(3,-1){30}}

\linethickness{0.05mm}
\put(-10,5){\vector(0,1){50}}

\thicklines
\put(30,0){\line(0,1){30}}
\put(10,30){\line(0,-1){20}}
\put(40,10){\line(0,1){30}}
\put(20,40){\line(0,-1){20}}
\put(0,20){\line(0,-1){20}}
\end{picture}
\qquad\qquad
\begin{picture}(50,55)(-10,0)

\thicklines
\put(30,0){\line(0,1){30}}
\put(10,30){\line(0,-1){20}}
\put(40,10){\line(0,1){30}}
\put(20,40){\line(0,-1){20}}
\put(0,20){\line(0,-1){20}}
%
\put(20,40){\line(1,0){20}}
\put(10,30){\line(1,0){ 8}} \put(22,30){\line(1,0){ 8}}
\put( 0,20){\line(1,0){ 8}} \put(12,20){\line(1,0){ 8}}
\put(10,10){\line(1,0){18}} \put(32,10){\line(1,0){ 8}}
\put( 0, 0){\line(1,0){30}}
\end{picture}
\caption{An arc presentation of a trefoil knot and its grid diagram}\label{fig:trefoil}
\end{figure}

For a link $L$, let $c(L)$, $F_L(v,z)$, and $\operatorname{spr}_v(F_L(v,z))$ denote the minimal crossing number, the Kauffman polynomial, and the $v$-spread of $F_L(v,z)$, i.e., the difference between the highest degree and the lowest degree of the variable $v$ in $F_L(v,z)$, respectively.
Here we list some of the important known results about the arc index.

\begin{proposition}[Cromwell]\label{prop:arc presentation exists}
Every link admits an arc presentation.
\end{proposition}

\begin{theorem}[Cromwell]\label{thm:additivity}
If $L_1$ and $L_2$ are nontrivial links, then
\begin{equation}\label{eq:alpha-2 is additive}\notag
\alpha(L_1\sharp L_2)=\alpha(L_1)+\alpha(L_2)-2
\end{equation}
\end{theorem}

\begin{theorem}[Bae-Park\footnote{Chan-Young Park}]\label{thm:2000, Bae-Park}
If $L$ is a nonsplit link, then
\begin{equation}\label{eq:alpha<=c+2}\notag
\alpha(L)\le c(L) + 2
\end{equation}
\end{theorem}

\begin{theorem}[Morton-Beltrami]\label{thm:1998, Morton-Beltrami}
For every link $L$, we have
\begin{equation}\label{eq:alpha>=spread+2}\notag
\alpha(L) \ge \operatorname{spr}_v(F_L(v,z)) + 2
\end{equation}
In particular, if $L$ is an alternating link, then
\begin{equation}\label{eq:alpha>=c+2}\notag
\alpha(L) \ge c(L) + 2
\end{equation}
\end{theorem}

\begin{theorem}[Jin-Park\footnote{Wang Keun Park}]\label{thm:2007, Jin-Park} A prime link $L$ is nonalternating if and only if
\begin{equation}\label{eq:alpha<=c}\notag
\alpha(L)\le c(L)
\end{equation}
\end{theorem}

Theorem~\ref{thm:additivity} allows us to focus on prime links.
Theorem~\ref{thm:2000, Bae-Park} and Theorem~\ref{thm:1998, Morton-Beltrami} together imply that the arc index equals the minimal crossing number plus two for nonsplit alternating links.

Theorem~\ref{thm:1998, Morton-Beltrami} and Theorem~\ref{thm:2007, Jin-Park} together imply Corollary~\ref{cor:spread+2<=alpha<=c} which leads us to conclude that, for prime nonalternating links, if the $v$-spread of the Kauffman polynomial plus two is equal to the minimal crossing number, then it is equal to the arc index.
\begin{corollary}\label{cor:spread+2<=alpha<=c}
A prime nonalternating link $L$ satisfies the inequality
\begin{equation}\label{eq:spread<=c-2}\notag
\operatorname{spr}_v(F_L(v,z))+2\le\alpha(L) \le c(L)
\end{equation}
\end{corollary}
Table~\ref{tab:spread+2=xing} shows the number of prime nonalternatings knots up to 16 crossings and those satisfying both equalities in Corollary~\ref{cor:spread+2<=alpha<=c}.

\begin{table}[h]\label{tab:spread+2=xing}
\begin{tabular}{|c|c|c|}
\hline
minimal crossing & prime nonalternating  & prime nonalternating knots with \\
 number $n$& knots with $n$ crossings & $n$ crossings and $v\text{-spread}+2=n$ \\ \hline
\hphantom{0}8 & \hphantom{0,000,00}3 & \hphantom{000,00}2 \\
\hphantom{0}9 & \hphantom{0,000,00}8 & \hphantom{000,00}6  \\
           10 & \hphantom{0,000,0}42 & \hphantom{000,0}32 \\
           11 & \hphantom{0,000,}185 & \hphantom{000,}135  \\
           12 & \hphantom{0,000,}888 & \hphantom{000,}627  \\
           13 & \hphantom{0,00}5,110 & \hphantom{00}3,250  \\
           14 & \hphantom{0,0}27,436 & \hphantom{0}15,735  \\
           15 & \hphantom{0,}168,030 & \hphantom{0}83,106  \\
           16 & \hphantom{}1,008,906 & \hphantom{}423,263 \\ \hline
\end{tabular}
\smallskip
\caption{Nonalternating primes knots whose arc index is determined by Corollary~\ref{cor:spread+2<=alpha<=c}}
\end{table}

In this article, we give three conditions for diagrams of a knot or link to have the arc index smaller than the number of crossings. For each of these conditions we give a list of 13 crossing knots satisfying the condition and having the arc index~12.

\section{The Knot-spoke diagram approach due to Bae and Park}

Now we briefly describe the methods used in the proofs of Theorems \ref{thm:2000, Bae-Park} and \ref{thm:2007, Jin-Park}.

A \emph{{wheel diagram\/}} is a finite plane graph of straight edges which are incident to a single vertex~\cite{BP2000}. The projection of an arc presentation of a knot or a link into the $xy$-plane is of this shape.
For a wheel diagram with $n$ edges to represent a knot or a link, each edge is labeled with an unordered pair of distinct integers, $1,2,\ldots,n$, so that each of the integers appears exactly twice. These numbers indicate the relative $z$-levels of the end points of the corresponding arcs. Since there are only finitely many choices for labelings, there are only finitely many knots and links for each arc index.

\begin{figure}[h]
\centering
\begin{picture}(200,60)(-25,-25)
\thicklines
\put(0,0){\line(1,0){25}}
\put(0,0){\line(-1,0){25}}
\put(0,0){\line(2,3){15}}
\put(0,0){\line(2,-3){15}}
\put(0,0){\line(-2,3){15}}
\put(0,0){\line(-2,-3){15}}
\put(0,0){\circle*{5}}
\small
\put(150,0){\line(1,0){25}}   \put(177,-3){{3,5}}
\put(150,0){\line(2,3){15}}   \put(165,24){{1,4}}
\put(150,0){\line(-2,3){15}}  \put(123,24){{2,6}}
\put(150,0){\line(-1,0){25}}  \put(112,-3){{1,5}}
\put(150,0){\line(-2,-3){15}} \put(123,-29){{3,6}}
\put(150,0){\line(2,-3){15}}  \put(165,-29){{2,4}}
\put(150,0){\circle*{5}}
\end{picture}
\caption{Wheel Diagrams}
\end{figure}
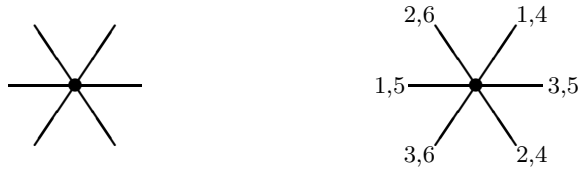

\newcounter{enumicnt}
A \emph{{knot-spoke diagram}} $D$ is a finite connected plane graph satisfying the two conditions:
\begin{enumerate}
\item There are three kinds of vertices in $D$; \emph{{a distinguished vertex\/}} $v_0$ with valency at least four, 4-valent vertices, and 1-valent vertices.
\item Every edge incident to a 1-valent vertex is also incident to $v_0$. Such an edge is called a \emph{{spoke}}.
\setcounter{enumicnt}{\theenumi}
\end{enumerate}
A wheel diagram is a knot-spoke diagram without any non-spoke edges.
A knot-spoke diagram $D$ is said to be \emph{{prime}} if no simple closed curve meeting $D$ in two interior points of edges separates multi-valent vertices into two parts.
A multi-valent vertex $v$ of a knot-spoke diagram $D$ is said to be a \emph{{cut-point}} if there is a simple closed curve $S$ meeting $D$ in the single point $v$ and separating non-spoke edges into two parts.

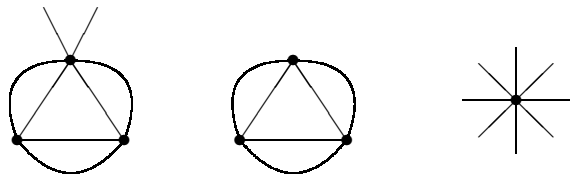
\begin{figure}[h]
{
\centering
\begin{picture}(80,70)(0,10)
\put(40,55){\circle*{4}} \put(20,25){\circle*{4}} \put(60,25){\circle*{4}}
{
\put(20,25){\line(1,0){40}} \put(20,25){\line(2,3){20}} \put(60,25){\line(-2,3){20}}
\qbezier(20,25)(40,0)(60,25) \qbezier(20,25)(9,55)(40,55) \qbezier(60,25)(71,55)(40,55)
\put(40,55){\line(-1,2){10}} \put(40,55){\line(1,2){10}}
}
\end{picture}
\begin{picture}(80,70)(0,10)
\put(40,55){\circle*{4}} \put(20,25){\circle*{4}} \put(60,25){\circle*{4}}
{
\put(20,25){\line(1,0){40}} \put(20,25){\line(2,3){20}} \put(60,25){\line(-2,3){20}}
\qbezier(20,25)(40,0)(60,25) \qbezier(20,25)(9,55)(40,55) \qbezier(60,25)(71,55)(40,55)
}
\end{picture}
\begin{picture}(80,70)(0,10)
\put(40,40){\circle*{4}}
{
\put(40,40){\line(1,0){20}}\put(40,40){\line(-1,0){20}}
\put(40,40){\line(0,1){20}}\put(40,40){\line(0,-1){20}}
\put(40,40){\line(1,1){14}}\put(40,40){\line(-1,1){14}}
\put(40,40){\line(1,-1){14}}\put(40,40){\line(-1,-1){14}}
}
\end{picture}
\caption{Knot-spoke diagrams}
}
\end{figure}
\begin{figure}[h]
{
\centering
\begin{picture}(80,70)(0,10)
\put(40,55){\circle*{4}} \put(20,25){\circle*{4}} \put(60,25){\circle*{4}}
{
\put(20,25){\line(1,0){40}} \put(20,25){\line(2,3){20}} \put(60,25){\line(-2,3){20}}
\qbezier(20,25)(40,0)(60,25) \qbezier(20,25)(9,55)(40,55) \qbezier(60,25)(71,55)(40,55)
\put(40,55){\line(-1,2){10}} \put(40,55){\line(1,2){10}}
}
\end{picture}
\begin{picture}(130,70)(5,10)
\put(40,55){\circle*{4}} \put(20,25){\circle*{4}} \put(60,25){\circle*{4}}
\put(100,25){\circle*{4}} \put(80,55){\circle*{4}} \put(120,55){\circle*{4}}
{
\put(20,25){\line(1,0){80}} \put(20,25){\line(2,3){20}} \put(60,25){\line(-2,3){20}}
\qbezier(20,25)(40,0)(60,25) \qbezier(20,25)(9,55)(40,55) 
\put(40,55){\line(-1,2){10}} \put(40,55){\line(1,2){10}}
\put(100,25){\line(-2,3){20}} \put(100,25){\line(2,3){20}} \put(40,55){\line(1,0){80}}
\qbezier(120,55)(131,26)(100,25) \qbezier(80,55)(100,80)(120,55)
}
\end{picture}
\caption{Prime diagram and non-prime diagram}
}
\end{figure}
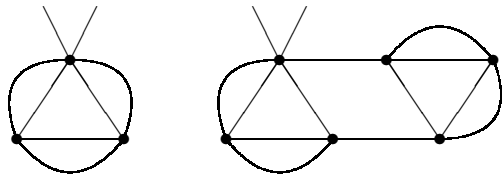
\begin{figure}[h]
{
\centering
\begin{picture}(100,70)(-50,10)
{
\qbezier(-40,40)(-20,30)(0,40) \qbezier(-40,40)(-20,50)(0,40)
\qbezier(-40,40)(-60,-5)(0,40) \qbezier(-40,40)(-60,85)(0,40)
\qbezier( 40,40)( 20,30)(0,40) \qbezier( 40,40)( 20,50)(0,40)
\qbezier( 40,40)( 60,-5)(0,40) \qbezier( 40,40)( 60,85)(0,40)
\put(0,40){\line(-2,3){20}} \put(0,40){\line( 2,3){20}}
}
\put(-40,40){\circle*{2}} \put(0,40){\circle*{4}} \put(40,40){\circle*{2}}
\put(-2,30){$v$}
\end{picture}
\caption{A cut-point}
}
\end{figure}
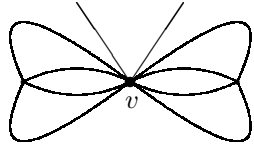

Notice that a cut-point free knot-spoke diagram with more than one non-spoke edges cannot have a loop, and that if a prime knot-spoke diagram $D$ has a cut-point, then the distinguished vertex $v_0$ must be the cut-point with valency bigger than four.

To obtain types of a knot or a link which can be projected onto a knot-spoke diagram $D$, we need to assign relative heights of the endpoints of edges of $D$ in the following way.
\begin{enumerate}
\addtocounter{enumi}{\theenumicnt}
\item At every 4-valent vertex, pairs of opposite edges meet in two distinct levels so that a knot-crossing is created.
\item If the distinguished vertex $v_0$ is incident to $2a$ non-spoke edges and $b$ spokes, then its small neighborhood  is the projection of $n=a+b$ arcs at distinct levels whose relative $z$-levels can be specified by the numbers $1,\cdots,n$. Every spoke is understood as the projection of an arc on a vertical plane whose endpoints project to $v_0$.
\setcounter{enumicnt}{\theenumi}
\end{enumerate}

\begin{figure}[h]
\centering
\begin{picture}(120,65)
\small
\put(30,30){\circle*{4}} \put(90,30){\circle*{4}}
\put(29,34){$v_0$} \put(59,34){$e$} \put(92,34){$v_1$}
\put(80,10){$\bar e$} \put(80,45){$\bar e^\prime$} \put(115,34){$e^\prime$}
{
\put(30,30){\line(-1, 2){10}}
\put(30,30){\line(-1, 1){20}}
\put(30,30){\line(-3, 2){30}}
\put(30,30){\line(-5, 1){30}}
\put(30,30){\line(-1, 0){30}}
\put(30,30){\line(-4,-1){30}}
\put(30,30){\line(-3,-2){30}}
\put(30,30){\line(-1,-1){20}}
\put(30,30){\line(-1,-2){10}}
\put(30,30){\line( 1, 0){60}}
\put(90,30){\line(1, 0){30}}
\put(90,30){\line(0, 1){30}}
\put(90,30){\line(0,-1){30}}
}
\end{picture}
\qquad
\begin{picture}(120,65)
\small
\put(30,30){\circle*{4}}
\put(29,34){$v_0$} \put(105,34){$e^\prime$}
\put(65,10){$\bar e$} \put(65,45){$\bar e^\prime$}
{
\put(30,30){\line(-1, 2){10}}
\put(30,30){\line(-1, 1){20}}
\put(30,30){\line(-3, 2){30}}
\put(30,30){\line(-5, 1){30}}
\put(30,30){\line(-1, 0){30}}
\put(30,30){\line(-4,-1){30}}
\put(30,30){\line(-3,-2){30}}
\put(30,30){\line(-1,-1){20}}
\put(30,30){\line(-1,-2){10}}
\put(30,30){\line( 1, 0){60}}
\qbezier(30,30)(85,40)(90,60)
\put(90,30){\line(1, 0){30}}
\qbezier(30,30)(85,20)(90, 0)
}
\end{picture}
\centerline{\small Local diagram of $D$ near $e$
\qquad\qquad
Local diagram of $D_e$ near $v_0$}
\caption{Contraction of an edge in $D$}\label{fig:contraction}
\end{figure}
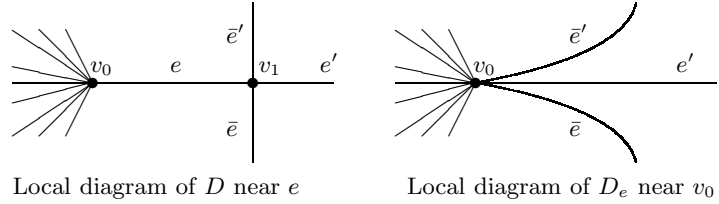

Let $e$ be an edge of a cut-point free knot-spoke diagram $D$ as in Figure~\ref{fig:contraction}. The knot-spoke diagram $D_e$ is obtained by contracting $e$ and then replacing each simple loop created from $\bar e$ or $\bar e^\prime$ by a spoke. The relative $z$-levels of the edges $e^\prime, \bar e, \bar e^\prime$ at $v_0$ in $D_e$ are easily decided by the $z$-level of $e$ at $v_0$ and the type of the crossing $v_1$ so that we do not need to keep track of the $z$-levels in detail for the proof of Theorem~\ref{thm:2000, Bae-Park}. But for the proof of Theorem~\ref{thm:2007, Jin-Park} we need to pay attention to some spokes corresponding to nonalternating edges.

\begin{lemma}[Bae-Park]\label{lem:no cut-point}
Let $D$ be a prime knot-spoke diagram without cut-points. Suppose that $D$ has at least two multi-valent vertices. Then there are at least two non-loop non-spoke edges e and f, incident to $v_0$, such that the knot-spoke diagrams $D_e$ and $D_f$ have no cut-points.
\end{lemma}

A loop in a knot-spoke diagram is said to be \emph{{simple}\/} if the other non-spoke edges are in one side of it.
By the above lemma, the edge contractions can be performed repeatedly, without creating a cut-point, until we obtain a knot-spoke diagram with $c(D)$ spokes and only one non-spopke edge which is a non-simple loop where $c(D)$ is the number of crossings in $D$. Notice that the following three properties are preserved.
\begin{enumerate}
\addtocounter{enumi}{\theenumicnt}
\item $D$ and $D_e$ represent the same knot or link.
\item \label{cond:region+spoke=c+2} The sum of the number of regions divided by the non-spoke edges and the number of spokes is unchanged.
\item $D_e$ is prime if $D$ is prime.
\end{enumerate}
The last non-spoke edge, which is a loop, is being folded to create two extra spokes to show the inequality $\alpha(L)\le c(L)+2$ of Theorem~\ref{thm:2000, Bae-Park}.

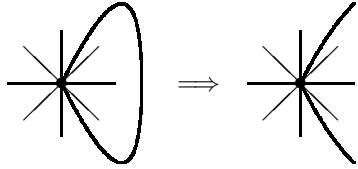
\begin{figure}[h]
\centering
\begin{picture}(50,80)(20,00)
\put(40,40){\circle*{4}}
{
\put(40,40){\line(1,0){20}}\put(40,40){\line(-1,0){20}}
\put(40,40){\line(0,1){20}}\put(40,40){\line(0,-1){20}}
\put(40,40){\line(1,1){14}}\put(40,40){\line(-1,1){14}}
\put(40,40){\line(1,-1){14}}\put(40,40){\line(-1,-1){14}}
{\thicklines\qbezier(40,40)(70,100)(70,40)\qbezier(40,40)(70,-20)(70,40)}
}
\end{picture}
\quad\raise37pt\hbox{$\Longrightarrow$}\quad
\begin{picture}(50,80)(20,00)
\put(40,40){\circle*{4}}
{
\put(40,40){\line(1,0){20}}\put(40,40){\line(-1,0){20}}
\put(40,40){\line(0,1){20}}\put(40,40){\line(0,-1){20}}
\put(40,40){\line(1,1){14}}\put(40,40){\line(-1,1){14}}
\put(40,40){\line(1,-1){14}}\put(40,40){\line(-1,-1){14}}
{\thicklines
\qbezier(40,40)(50,60)(60,70)\qbezier(40,40)(50,20)(60,10)
}
}
\end{picture}
\caption{Folding the last non-spoke edge}
\end{figure}

In the case of nonalternating diagrams, there are at least two \emph{removable\/} spokes so that the inequality of Theorem~\ref{thm:2007, Jin-Park} can be proved. The edges to be contracted must be chosen carefully to make nonalternating edges into removable spokes. Therefore a more elaborate method than Lemma~\ref{lem:no cut-point} is needed to avoid cut-point. The following lemma plays an important role for this purpose.

\begin{lemma}[Jin-Park]\label{lem:Jin-Park}
Let $D$ be a prime cut-point free knot-spoke diagram and let $e$ be an edge incident to $v_0$ and to another $4$-valent vertex $v_1$ such that $D_e$ has a cut-point. Then there exists a simple closed curve $S_e$ satisfying the following conditions.
\begin{enumerate}
\item $D_e \cap S_e = v_0$
\item $S_e$ separates $\bar e$ and $\bar e^\prime$ where the four edges incident to $v_1$ in $D$ are labeled with $e, \bar e, e^\prime, \bar e^\prime$ as in Figure~\ref{fig:contraction}.
\item $S_e$ separates $D_e$ into two knot-spoke diagrams $\bar D$ and $\bar D^\prime$ containing $\bar e$ and $\bar e^\prime$, respectively. Futhermore $\bar D^\prime$ is prime and cut-point free, and there is a sequence of non-spoke edges $e_1, e_2, \cdots, e_k$ of $D$ not contained in $\bar D^\prime$ such that the knot-spoke diagram $D_{{e_1}{e_2}{\cdots}{e_k}}$ is identical with $\bar D^\prime$ on non-spoke edges in one side of $S_e$ and has only spokes in the other side.
\end{enumerate}
\end{lemma}


\section{Filtered Spanning Trees}

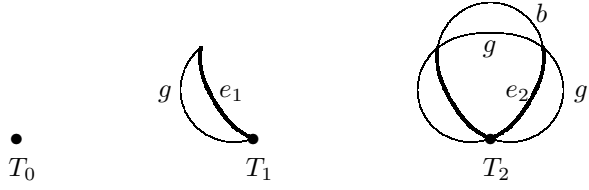
\begin{figure}[t]
\centering
\setlength{\unitlength}{0.05cm}
\newcommand{\mythicklines}{\linethickness {0.30mm}}
\newcommand{\mythinlines}{\linethickness {0.05mm}}
\begin{picture}(60,40)(-30,-20)
\put(0.,-16.){\circle*{3}}
\put(-2,-26){$T_0$}
\end{picture}
\begin{picture}(60,40)(-30,-20)
\mythinlines
\qbezier(0.,-16.)(-4.717259301,-17.76897224)(-9.414213562,-16.30589621)
\qbezier(-9.414213562,-16.30589621)(-14.61513958,-14.68583569)(-17.32050808,-10.)
\qbezier(-17.32050808,-10.)(-20.02587658,-5.314164317)(-18.82842712,0.)
\qbezier(-18.82842712,0.)(-17.74701097,4.799219772)(-13.85640646,8.)
\mythicklines
\qbezier(0.,-16.)(-3.676759433,-14.62121521)(-6.585786438,-11.40691672)
\qbezier(-6.585786438,-11.40691672)(-8.477262958,-9.316949831)(-10.39230485,-6.)
\qbezier(-10.39230485,-6.)(-12.30734672,-2.683050212)(-13.17157288,0.)
\qbezier(-13.17157288,0.)(-14.50072354,4.126440629)(-13.85640646,8.)
\put(0.,-16.){\circle*{3}}
\put(-25,-5){$g$}
\put(-9,-5){$e_1$}
\put(-2,-26){$T_1$}
\end{picture}
\begin{picture}(60,40)(-30,-20)
\mythinlines
\qbezier(-13.85640646,8.)(-13.02975166,12.96975254)(-9.414213562,16.30589621)
\qbezier(-9.414213562,16.30589621)(-5.410736959,20.00000003)(0.,20.)
\qbezier(0.,20.)(5.410737006,20.)(9.414213562,16.30589621)
\qbezier(9.414213562,16.30589621)(13.02975171,12.96975249)(13.85640646,8.)
\mythicklines
\qbezier(13.85640646,8.)(14.50072352,4.126440545)(13.17157288,0.)
\qbezier(13.17157288,0.)(12.30734671,-2.683050203)(10.39230485,-6.)
\qbezier(10.39230485,-6.)(8.477262982,-9.316949809)(6.585786438,-11.40691672)
\qbezier(6.585786438,-11.40691672)(3.676759418,-14.62121522)(0.,-16.)
\mythinlines
\qbezier(0.,-16.)(-4.717259301,-17.76897224)(-9.414213562,-16.30589621)
\qbezier(-9.414213562,-16.30589621)(-14.61513958,-14.68583569)(-17.32050808,-10.)
\qbezier(-17.32050808,-10.)(-20.02587658,-5.314164317)(-18.82842712,0.)
\qbezier(-18.82842712,0.)(-17.74701097,4.799219772)(-13.85640646,8.)
\qbezier(-13.85640646,8.)(-10.82396409,10.49477467)(-6.585786438,11.40691672)
\qbezier(-6.585786438,11.40691672)(-3.830083749,12.00000000)(0.,12.)
\qbezier(0.,12.)(3.830083757,12.00000000)(6.585786438,11.40691672)
\qbezier(6.585786438,11.40691672)(10.82396411,10.49477466)(13.85640646,8.)
\qbezier(13.85640646,8.)(17.74701098,4.799219749)(18.82842712,0.)
\qbezier(18.82842712,0.)(20.02587658,-5.314164357)(17.32050808,-10.)
\qbezier(17.32050808,-10.)(14.61513954,-14.68583572)(9.414213562,-16.30589621)
\qbezier(9.414213562,-16.30589621)(4.717259271,-17.76897223)(0.,-16.)
\mythicklines
\qbezier(0.,-16.)(-3.676759433,-14.62121521)(-6.585786438,-11.40691672)
\qbezier(-6.585786438,-11.40691672)(-8.477262958,-9.316949831)(-10.39230485,-6.)
\qbezier(-10.39230485,-6.)(-12.30734672,-2.683050212)(-13.17157288,0.)
\qbezier(-13.17157288,0.)(-14.50072354,4.126440629)(-13.85640646,8.)
\put(0.,-16.){\circle*{3}}
\put(-2,7){$g$}
\put(4,-5){$e_2$}
\put(22,-5){$g$}
\put(12,15){$b$}
\put(-2,-26){$T_2$}
\end{picture}
\caption{A filtered spanning tree and its successive closures}\label{fig:filtered trefoil diagram}
\end{figure}

\begin{figure}[t]
\centering
\begin{picture}(60,60)(0,7)
{\thicklines\put(0,53){\line(1,0){60}}
}
\put(33,11){\line(-1,2){19}} \qbezier(10,57)(8,61)(6,65)
\put(27,11){\line( 1,2){19}} \qbezier(50,57)(52,61)(54,65)
\put(25,17){\line(1,0){10}}
\put(12,53){\circle*{4}} \put(48,53){\circle*{4}} \put(30,17){\circle*{4}}\put(28,5){$v_0$}
\end{picture}%
\qquad\qquad
\begin{picture}(60,60)(0,7)
{\thicklines\qbezier(0,53)(15,40)(30,17)\qbezier(30,17)(45,40)(60,53)}
\put(33,11){\line(-1,2){27}}
\put(27,11){\line( 1,2){27}}
\put(25,17){\line(1,0){10}}
\put(30,17){\circle*{4}}\put(28,5){$v_0$}
\end{picture}
\caption{Contraction of a triangular region using a doubly-good edge}\label{fig:doubly-good}
\end{figure}
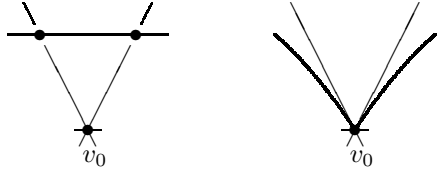

Instead of collapsing edges of a diagram $D$ in sequence to obtain a wheel diagram, we consider the tree in $D$ consisting of the edges to be contracted. With this new approach, we describe the method used in the proof of Theorem~\ref{thm:2007, Jin-Park}.

Let $D$ be a knot diagram. We may consider $D$ as a connected $4$-valent plane graph with $c(D)$ vertices and $2\,c(D)$ edges.
A \emph{filtered tree\/} of $D$ is an increasing sequence
$T_0\subset T_1\subset T_2\subset\cdots\subset T_k$
of subgraphs of $D$ such that each $T_i$ is a tree containing $i$ edges. The edges of $T_k$ are ordered by the filtering.
On the other hand, if the edges of a tree $T$ are ordered so that each of their successive unions is connected, the ordering gives rise to a filtered tree structure on $T$. If a tree $T$ is prescribed with such an ordering we can consider $T$ as a filtered tree.

The \emph{{closure}} of $T_i$, denoted by $\overline T_i$, is the subgraph of $D$ obtained from $T_i$ by adding the edges which are incident to $T_i$ at both ends.
An edge $f$ of $\overline T_i\setminus T_{i}\subset D$ is said to be \emph{good} if it meets the edge $e_i=T_i\setminus T_{i-1}$ transversely at the vertex not contained in $T_{i-1}$.
An edge $f$ of $\overline T_i\setminus T_{i}\subset D$ is said to be \emph{bad} if it is an extension of the edge $e_i$ at the vertex not contained in $T_{i-1}$. In Figure~\ref{fig:filtered trefoil diagram}, good edges and bad edges are labeled with the letters $g$ and $b$, respectively.

Let $T_0\subset T_1\subset\cdots\subset T_i=T$  be a filtered tree in a diagram $D$ which does not
span $D$.
A simple arc $\Gamma$, which does not form a bigon together with a single edge of $D$, is called a \emph{cutting arc\/} of $T$ if $\Gamma\cap D$ consists of two
distinct vertices $p\in T_i\setminus T_{i-1}$ and $c\in T_{i-1}$ such that the simple closed curve in
$\Gamma\cup T_i$ separates edges of $D\setminus\overline T_i$ into two parts.
We say that the filtered tree $T$ is \emph{good\/} if, for each $j\le i$, the subtree $T_0\subset T_1\subset\cdots\subset T_j$ of $T$ has no cutting arc and $\overline T_j$ has no bad edge.

If a filtered tree of $D$ terminates with a spanning tree of $D$, we call it a \emph{filtered spanning tree\/} of $D$. As every spanning tree of $D$ has $c(D)-1$ edges, a filtered spanning tree of $D$ is of the form $T_0\subset T_1\subset T_2\subset\cdots\subset T_{c(D)-1}$.
A filtered spanning tree $T_0\subset T_1\subset T_2\subset\cdots\subset T_{c(D)-1}$ is said to be \emph{good\/} if
each $T_i$ is good for $i<c(D)-1$ and if $T_{c(D)-1}$ has no cutting arc. Notice that $\overline T_{c(D)-1}$ has a bad edge.

We rephraise the statement of Theorem~\ref{thm:2000, Bae-Park} in the following way.

\begin{theorem}[Theorem~\ref{thm:2000, Bae-Park} rephrased]
A prime link diagram $D$ admits a good filtered spanning tree and therefore one can obtain an arc presentation with $c(D)+2$ arcs.
\end{theorem}

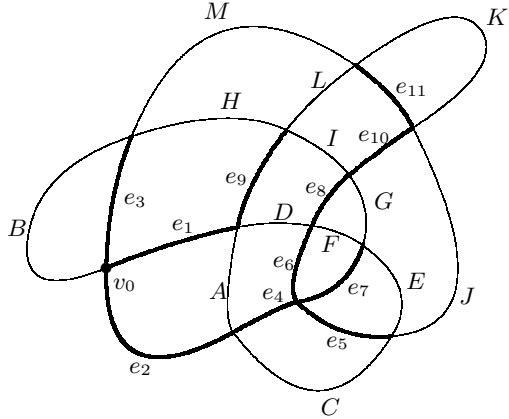
\begin{figure}[b]
\centering
\setlength{\unitlength}{0.07cm}
\begin{picture}(100,85)(-50,-35)
\small
\linethickness{0.2pt}
\qbezier(-44,2)(-42,15)(-24.57894736, 20.80701754)
\qbezier(-24.57894736, 20.80701754)(-6,27)(4.514563108, 22.01941748)
\put(-8,26){$H$}
\qbezier(4.514563108, 22.01941748)(13,18)(16.27461140, 13.49740932)
\put(12,19){$I$}
\qbezier(16.27461140, 13.49740932)(21,7)(18.86440678, .166101694)
\put(21,7){$G$}
\linethickness{1pt}
\qbezier(18.86440678, .166101694)(16,-9)(6.571428572, -10.57142857)
\put(16,-9){$e_7$}
\qbezier(6.571428572, -10.57142857)(4,-11)(-5.489677953, -16.30305533)
\put(0,-10){$e_4$}
\qbezier(-5.489677953, -16.30305533)(-30,-30)(-29.35278745, -4.111498254)
\put(-25,-24){$e_2$}
\qbezier(-29.35278745, -4.111498254)(-29,10)(-24.57894736, 20.80701754)
\put(-26,8){$e_3$}
\linethickness{0.2pt}
\qbezier(-24.57894736, 20.80701754)(-11,54)(17.43870968, 34.25089605)
\put(-11,43){$M$}
\linethickness{1pt}
\qbezier(17.43870968, 34.25089605)(25,29)(28.27868852, 22.44262295)
\put(25,29){$e_{11}$}
\linethickness{0.2pt}
\qbezier(28.27868852, 22.44262295)(47,-15)(24.25723472, -17.06752412)
\put(37,-11){$J$}
\linethickness{1pt}
\qbezier(24.25723472, -17.06752412)(14,-18)(6.571428572, -10.57142857)
\put(12,-19){$e_5$}
\qbezier(6.571428572, -10.57142857)(4,-8)(9.280000000, 3.880000000)
\put(2,-4){$e_6$}
\qbezier(9.280000000, 3.880000000)(12,10)(16.27461140, 13.49740932)
\put(8,10){$e_8$}
\linethickness{1pt}
\qbezier(16.27461140, 13.49740932)(23,19)(28.27868852, 22.44262295)
\put(18,20){$e_{10}$}
\linethickness{0.2pt}
\qbezier(28.27868852, 22.44262295)(46,34)(41,41)
\put(42,42){$K$}
\qbezier(41,41)(36,48)(17.43870968, 34.25089605)
\qbezier(17.43870968, 34.25089605)(9,28)(4.514563108, 22.01941748)
\put(9,30){$L$}
\linethickness{1pt}
\qbezier(4.514563108, 22.01941748)(-3,12)(-4.681318681, 3.593406592)
\put(-7,12){$e_{9}$}
\linethickness{0.2pt}
\qbezier(-4.681318681, 3.593406592)(-8,-13)(-5.489677953, -16.30305533)
\put(-10,-10){$A$}
\qbezier(-5.489677953, -16.30305533)(11,-38)(24.25723472, -17.06752412)
\put(11,-32){$C$}
\qbezier(24.25723472, -17.06752412)(30,-8)(18.86440678, .166101694)
\put(27,-8){$E$}
\qbezier(18.86440678, .166101694)(15,3)(9.280000000, 3.880000000)
\put(11,-1){$F$}
\qbezier(9.280000000, 3.880000000)(2,5)(-4.681318681, 3.593406592)
\put(2,5){$D$}
\linethickness{1pt}
\qbezier(-4.681318681, 3.593406592)(-17,1)(-29.35278745, -4.111498254)
\put(-17,3){$e_1$}
\linethickness{0.2pt}
\qbezier(-29.35278745, -4.111498254)(-46,-11)(-44,2)
\put(-29.35278745, -4.111498254){\circle*{2}}
\put(-28,-8){$v_0$}
\put(-48,2){$B$}
\end{picture}
\caption{Good edges and a bad edge}
\label{fig:good-tree}
\end{figure}

A good edge $e\subset\overline T_i\setminus\overline T_{i-1}$ is said to be \emph{doubly good\/}
if it is a nonalternating edge and the simple closed curve in $T_i\cup e$ has only good edges of $T_j$, $j\le i$, in one side.
A doubly good edge $e\subset\overline T_i\setminus\overline T_{i-1}$ and the two edges $e_i=T_i\setminus T_{i-1}$, and $e_{i-1}=T_{i-1}\setminus T_{i-2}$ together bound a nonalternating triangular region in $D/T_{i-2}$, as shown in Figure~\ref{fig:doubly-good}, which can be contracted to reduce the number of regions by one without increasing the number of spokes. Thus the existence of one doubly good edge leads to an arc presentation with one less arcs than the process described in the property (\ref{cond:region+spoke=c+2}) of page~\pageref{cond:region+spoke=c+2}.

\begin{theorem}[Theorem~\ref{thm:2007, Jin-Park} rephrased]
A prime nonalternating diagram $D$ of a link has a good filtered spanning tree which has at least two doubly good edges so that $D$ has an arc presentation with $c(D)$ arcs.
\end{theorem}

In Figure~\ref{fig:good-tree}, the sequence $e_1,e_2,\ldots,e_{11}$ gives rise to a filtered spanning tree $T_i=v_0\cup e_1\cup\cdots\cup e_i$, $i=0,\ldots,11$. The edges $A$ through $L$ are good and the edge $M$ is bad.
The three edges $A$, $F$ and $I$ are doubly good if they are nonalternating.

\medskip

The following proposition is immediate from the definition of good filtered trees.
\begin{proposition}\label{prop:cut-point}
Let $T_0\subset T_1\subset\cdots\subset T_m$ be a non-spanning good filtered tree in a prime diagram $D$.
Let $e$ be an edge in $D$ such that $T_m \cap e$ is a single vertex, so that $T_m \cup e$ is a tree.
If  $T_0\subset T_1\subset\cdots\subset T_m\subset (T_m\cup e)$ is not a good filtered tree, then $\overline {T_m \cup e}$ has a bad edge or a sufficiently small neighborhood of\/ $\overline {T_{m} \cup e}$ has disconnected exterior in $D$.
\end{proposition}

\begin{figure}[h]
\centering
\begin{picture}(120,65)
\small
\put(30,30){\circle*{2}} \put(90,30){\circle*{2}}
\put(24,34){$v$} \put(59,34){$e$} \put(92,34){$p$} \put(55,45){$R$}
\put(80,10){$\bar e$} \put(80,45){$\bar e^\prime$} \put(115,34){$e^\prime$}
\put(30,30){\line(0, 1){30}}
\put(30,30){\line(-1, 0){30}}
\put(30,30){\line(0,-1){30}}
\put(30,30){\line( 1, 0){60}}
\put(90,30){\line(1, 0){30}}
\put(90,30){\line(0, 1){30}}
\put(90,30){\line(0,-1){30}}
\end{picture}
\caption{Extending $T_m$ along $e$ in $D$}\label{fig:extending T_m}
\end{figure}
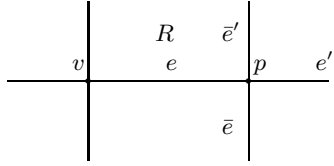

Suppose that $T_0\subset T_1\subset\cdots\subset T_m$ and $e$ are as in the hypothesis of Proposition~\ref{prop:cut-point} and that $T_0\subset T_1\subset\cdots\subset T_m\subset (T_m\cup e)$ is not a good filtered tree.
In Figure~\ref{fig:extending T_m}, $v$ is the vertex of $e$ belonging to $T_m$ and $p$ is the other vertex of $e$. The three edges $e^\prime$, $\bar e$, $\bar e^\prime$ are incident to $e$ at $p$ and $R$ is a region of $D$ whose boundary contains $e$ and $\bar e^\prime$. Proposition~\ref{prop:cut-point} implies that there are three cases to consider:
\begin{enumerate}
\item[(B1)] $e^\prime$ is a bad edge of $\overline{T_m\cup e}$ joining $p$ and a vertex $c$ of $T_m$.
\item[(B2)] There is a simple arc $\Gamma_p$ contained in a single region of $D$ joining $p$ and a vertex $c$ of $T_m$ such that the unique cycle $\overline\Gamma_p$ in $T_m\cup e\cup\Gamma_p$ does not enclose the region $R$ but encloses the edges $e^\prime$ and $\bar e$.
\item[(B3)] There is a simple arc $\Gamma_p$ as in (B2) except that $\overline\Gamma_p$ does not enclose $e^\prime$.
\end{enumerate}

\begin{figure}[h]
\centering
\begin{picture}(110,75)(10,-10)
\small
\put(30,30){\circle*{2}} \put(90,30){\circle*{2}} \put(120,30){\circle*{2}}
\put(24,34){$v$} \put(59,34){$e$} \put(92,34){$p$} \put(55,45){$R$}
\put(117,34){$c$}\put(107,8){$T_m$} \put(60,-10){(B1)}
\put(85,15){$\bar e$} \put(83,40){$\bar e^\prime$} \put(105,34){$e^\prime$}
\put(30,30){\line(0, 1){20}}
\put(30,30){\line(-1, 0){15}}
\put(30,30){\line(0,-1){15}}
\put(30,30){\line( 1, 0){60}}
\put(90,30){\line(1, 0){30}}
\put(90,30){\line(0, 1){20}}
\put(90,30){\line(0,-1){15}}
\qbezier[15](10,30)(10,17.5)(10,5)
\qbezier[75](10,5)(65,5)(120,5)
\qbezier[15](120,5)(120,17.5)(120,30)
\end{picture}\qquad
\begin{picture}(105,75)(10,-10)
\small
\put(10,30){\circle*{2}} \put(30,30){\circle*{2}} \put(90,30){\circle*{2}} 
\put(24,34){$v$} \put(59,34){$e$} \put(92,23){$p$} \put(55,45){$R$}
\put(8,34){$c$}\put(102,8){$\Gamma_p$} \put(55,-10){(B2)}
\put(85,15){$\bar e$} \put(83,40){$\bar e^\prime$} \put(105,32){$e^\prime$}
\put(30,30){\line(0, 1){20}}
\put(30,30){\line(-1, 0){15}}
\put(30,30){\line(0,-1){15}}
\put(30,30){\line( 1, 0){60}}
\put(90,30){\line(1, 0){20}}
\put(90,30){\line(0, 1){20}}
\put(90,30){\line(0,-1){15}}
\qbezier[15](10,30)(10,17.5)(10,5)
\qbezier[75](10,5)(62.5,5)(115,5)
\qbezier[25](115,5)(115,27.5)(115,50)
\qbezier[25](90,30)(102.5,40)(115,50)
\end{picture}\qquad
\begin{picture}(100,75)(10,-10)
\small
\put(10,30){\circle*{2}} \put(30,30){\circle*{2}} \put(90,30){\circle*{2}}
\put(24,34){$v$} \put(59,34){$e$} \put(92,34){$p$} \put(55,45){$R$}
\put(8,34){$c$}\put(102,8){$\Gamma_p$} \put(55,-10){(B3)}
\put(85,15){$\bar e$} \put(83,40){$\bar e^\prime$} \put(105,34){$e^\prime$}
\put(30,30){\line(0, 1){20}}
\put(30,30){\line(-1, 0){15}}
\put(30,30){\line(0,-1){15}}
\put(30,30){\line( 1, 0){60}}
\put(90,30){\line(1, 0){20}}
\put(90,30){\line(0, 1){20}}
\put(90,30){\line(0,-1){15}}
\qbezier[15](10,30)(10,17.5)(10,5)
\qbezier[75](10,5)(55,5)(100,5)
\qbezier[15](100,5)(95,17.5)(90,30)
\end{picture}
\caption{Three bad cases for $T_m\cup e$}\label{fig:3 bad cases}
\end{figure}
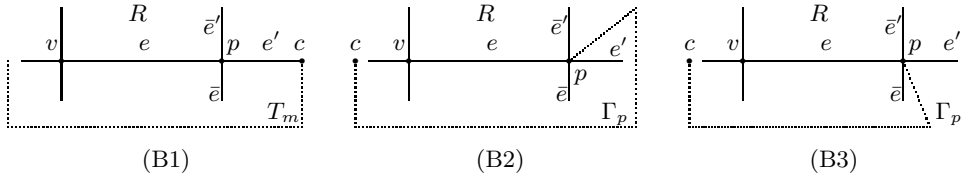

For each of the above cases, we say that $T_m\cup e$ is a (B$n$)-\emph{extension\/} of $T_m$, $n=1,2,3$. If $T_0\subset T_1\subset\cdots\subset T_m\subset (T_m\cup e)$ is a good filtered (spanning) tree, we say that $T_m\cup e$ is a \emph{good extension\/} of $T_m$.

\section{Main Theorems}
Before we state our main theorems, we give several lemmas and corollaries. The first two lemmas are translations of the two lemmas written in the language of knot-spoke diagrams into the ones written in the language of filtered trees.
\begin{lemma}[Lemma~\ref{lem:no cut-point} rephrased]\label{lem:Bae-Park-tree}
Let $D$ be a prime knot diagram with $c(D)$ crossings and let $T_0\subset T_1\subset\cdots\subset T_m$ be a non-spanning good filtered tree in $D$.
Then there are two edges $e$ and $f$ in $D\setminus\overline T_m$ such that $T_m\cup e$ and $T_m\cup f$ are good extensions of $T_m$.
\end{lemma}

\begin{lemma}[Lemma~\ref{lem:Jin-Park} rephrased]\label{lem:Jin-Park-tree}
Let $T_0\subset T_1\subset\cdots\subset T_m$ be a good filtered tree in a prime diagram $D$ with $m<c(D)-2$.
Let $e$ be an edge in $D$ such that $T_m \cap e$ is a single vertex, say $v$.
Suppose that $T_m\cup e$ is not a good extension of $T_m$.
Then there exists a simple closed curve $S$ satisfying the following conditions.
\begin{enumerate}
\item $D \cap S$ is a simple arc which is the union of $e$ and some edges of $T_m$.
\item $S$ separates $\bar e$ and $\bar e^\prime$ where the four edges incident to $p$, the endpoint of $e$ other than $v$, are labeled with $e, \bar e, e^\prime, \bar e^\prime$ as in Figure~\ref{fig:extending T_m}.
\item $S$ separates $D$ into two subgraphs $\bar D$ and $\bar D^\prime$ containing $\bar e$ and $\bar e^\prime$, respectively and satisfying $\bar D \cap \bar D^\prime = D \cap S$. Futhermore there is a sequence $e_1, e_2, \cdots, e_k$ of $D\setminus\ (\bar D^\prime \cup \overline T_m)$ such that
    $T_0 \subset \cdots \subset{T_m} \subset{T_m \cup e_1} \subset \cdots \subset{T_m \cup e_1 \cup \cdots \cup e_k} $ is a good filtered tree and $D \setminus {\overline{T_m \cup e_1 \cdots \cup e_k}} = \bar D^\prime \setminus \overline{T_m \cup e}$.
\end{enumerate}
\end{lemma}

\begingroup
\noindent{\bf Remark.} If $T_m \cup e$ is a (B3)-extension of $T_m$, the sequence $e_1, e_2, \cdots, e_k$ of Lemma~\ref{lem:Jin-Park-tree} can be chosen so that $e_k=e$. See the original proof of Lemma~\ref{lem:Jin-Park}~\cite[Proposition 8]{Jin-Park2007}.
\endgroup

\medskip
The closure of a region divided by a diagram $D$ is called a \emph{face\/} of $D$. Let $T_0\subset T_1\subset T_2\subset\cdots\subset T_j$ be a non-spanning filtered tree  of $D$. If $T_j$ has a cutting arc $\Gamma_p$, we may assume that it is \emph{innermost}, in the following sense: Let $F$ be the face of $D$ containing $\Gamma_p$ and let $\Delta\subset F$ be the disk enclosed by the unique cycle of $T_j\cup\Gamma_p$ satisfying $\partial\Delta\subset\partial F\cup\Gamma_p$. Then any cutting arc of $T_j$ contained in $\Delta$ is isotopic to $\Gamma_p$.

The following lemma asserts that two innermost cutting arcs are essentially disjoint. We omit the proof.

\begin{lemma}\label{lem:two cutting arcs}
Let $T_0\subset T_1\subset \cdots\subset T_j$ be a non-spanning filtered tree of $D$.
If $\Gamma_p$ and $\Gamma_q$ are innermost cutting arcs of $T_i$ and $T_j$, respectively, for some $i<j$, then we can isotope $\Gamma_p$ and $\Gamma_q$ so that they do not intersect in their interiors.
\end{lemma}

If $p$ and $q$ are the two vertices of an edge $e$ of $D$, we write $e=\overline{pq}$, even in the case that there is another edge joining $p$ and $q$ if we understand which is $e$.

\begin{lemma}\label{lem:step 2}
 Let $T_0\subset T_1\subset \cdots\subset T_m$ be a non-spanning good filtered tree of $D$ and let $\overline{sp}$, $\overline{pq}$, and $\overline{qr}$ be three consecutive edges along the boundary $\partial F$ of a face $F$ of $D$. Suppose that
 $\partial F\cap T_m\ne\emptyset$ and $\overline{pq}\cap T_m=\emptyset$.
 Then we can always construct a sequence of successive good extensions $T_{m+1}\subset \cdots\subset T_{m+k}$ of $T_m$ such that the closure $\overline T_{m+k}$ contains all edges of $\partial F$ except $\overline{sp}$,
$\overline{pq}$, and $\overline{qr}$.
\end{lemma}

\begin{proof}
We construct a sequence of successively extended filtered tree $T_{m+1}\subset \cdots\subset T_{m+k}$ of $T_m$ along the edges of $\partial F$ so that $T_{m+k}$ contains all vertices of $\partial F$ except $p$ and $q$. If $T_{m+i}$ is not a good extension for some $i$, we apply Lemma~\ref{lem:Jin-Park-tree} to obtain a larger good filtered tree $T^\prime$ not containing the edges $\overline{sp}$, $\overline{pq}$, and $\overline{qr}$, and continue.
\end{proof}

The following lemma is an immediate consequence of related definitions.
\begin{lemma}\label{lem:d-good-2,3}
Let $T_0\subset T_1\subset \cdots\subset T_m$ be a non-spanning good filtered tree of $D$ and let $\overline{sp}$, $\overline{pq}$, and $\overline{qr}$ be three consecutive edges along the boundary $\partial F$ of a face $F$ of $D$.
Suppose that $\overline{pq}$ is a nonalternating edge and $\partial F\setminus\{\overline{sp},\overline{pq},\overline{qr}\}\subset \overline T_m$.
Then $\overline{pq}$ becomes a doubly good edge of $\overline T_{m+2}$ if the following two conditions are satisfied:
\begin{enumerate}
\item $T_{m+1}=T_m \cup {\overline{sp}}$ is a good extension of $T_m$.
\item $T_{m+2}=T_{m+1}\cup\overline{rq}$ is a good extension of $T_{m+1}$.
\end{enumerate}
\end{lemma}

\begin{corollary}\label{cor:d-good-2,3}
Let $T_m$, $F$, $\overline{sp}$, $\overline{pq}$, and $\overline{qr}$ be as in Lemma~\ref{lem:d-good-2,3}. Suppose that the two conditions below are satisfied:
\begin{enumerate}
\item $T_{m+1}=T_m \cup {\overline{sp}}$ is a (B3)-extension of $T_m$ so that there is a tree $T^\prime$ which has a good extension $T^\prime\cup\overline{sp}$ containing $T_{m+1}$. (See the remark following Lemma~\ref{lem:Jin-Park-tree}.)
\item ${T^\prime} \cup {\overline{sp}} \cup {\overline{rq}}$ is a good extension of ${T^\prime} \cup{\overline{sp}}$.
  \end{enumerate}
Then $\overline{pq}$ is a doubly good edge of $\overline{T^\prime \cup \overline{sp} \cup \overline{rq}}$.
\end{corollary}

\begin{lemma}\label{lem:d-good-1}
Let $F$ be a face of a minimal crossing diagram $D$ of a prime knot such that $\partial F$ contains a nonalternating edge $\overline{pq}$. We may label the vertices of $\partial F$ as $v_0,v_1,\ldots,v_{n-2}=q,v_{n-1}=p$, cyclically around $F$, for some $n\ge3$. Then there is a good filtered tree $T_0\subset\cdots\subset T_m$ such that $T_0=v_0$, $v_i\in T_m$, $i=0,\ldots,n-3$ and $T_{m+1}$ is a good extension of $T_m$ along $\overline{v_{n-3}q}\subset\partial F$. Furthermore if the extension $T_{m+2}$ of $T_{m+1}$ along $\overline{v_0p}\subset\partial F$ is not (B3), then $\overline{pq}$ is a doubly good edge of $\overline T_{m+2}$.
\end{lemma}

\begin{proof}
We extend $T_0$ repeatedly along the edges $\overline{v_{i-1}v_i}$, $i=1,\ldots,n-3$. These extensions are neither (B1) nor (B2) since $D$ is prime. If a (B3)-extension occurs, then, by Lemma~\ref{lem:Jin-Park-tree}, we can insert more edges before the extension to obtain a good extension along the same edge. Continuing in this manner, we obtain the good extension $T_{m+1}=T_m\cup\overline{v_{n-3}q}$ of $T_m$. Since $D$ is prime, the extension is neither (B1) nor (B2). This completes the proof.
\end{proof}

\begin{corollary}
Suppose that the hypothesis of Lemma~\ref{lem:d-good-1} holds and that $v_0$ and $p$ are two vertices of a bigonal face adjacent to $F$.
Then $\overline{pq}$ is a doubly good edge of $\overline T_{m+2}$.
\end{corollary}

\begin{proof} In this case, the extension $T_{m+2}$ of $T_{m+1}$ mentioned in Lemma~\ref{lem:d-good-1} cannot be (B3).
\end{proof}

Let $n \geq 2$. An \emph{$(n,1)$-tangle\/} is an alternating tangle diagram of $n+1$ crossings whose projection is as shown in Figure~\ref{fig:non-alt-tangle} $(a)$. A nonalternating knot diagram $D$ is said to be \emph{$(n,1)$-nonalternating\/} if it can be decomposed of two alternating tangles one of which is an $(n,1)$-tangle.
Let $n \geq 1$. We can define an \emph{$n$-tangle\/} and \emph{$n$-nonalternating\/} diagram in a similar manner, using Figure~\ref{fig:non-alt-tangle} $(b)$.
A $1$-tangle is a single crossing and a $1$-nonalternating diagram is also called an \emph{almost alternating diagram}.

  \begin{figure}[h]
    \begin{center}
\includegraphics[scale=0.2]{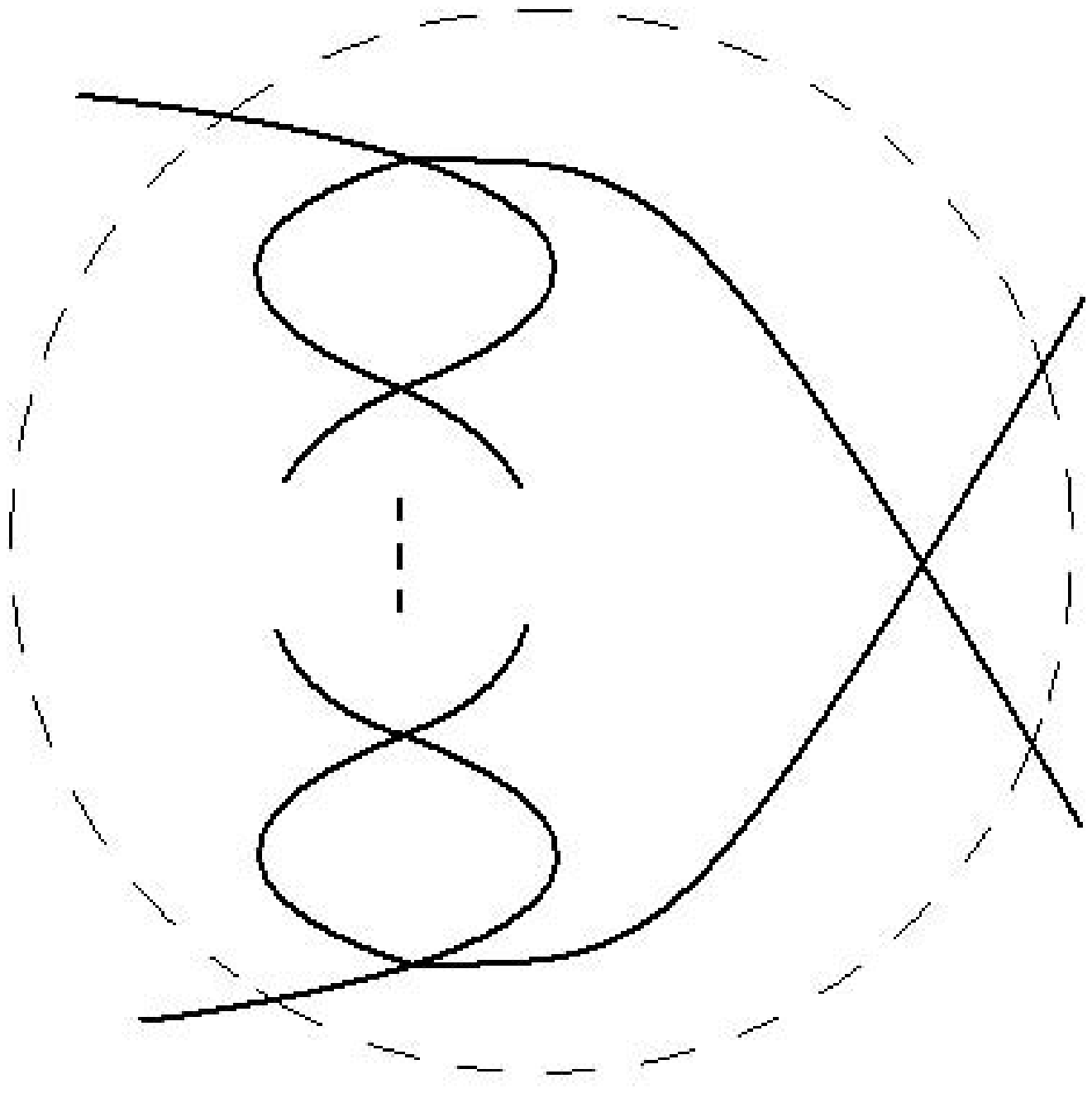} \qquad \includegraphics[scale=0.2]{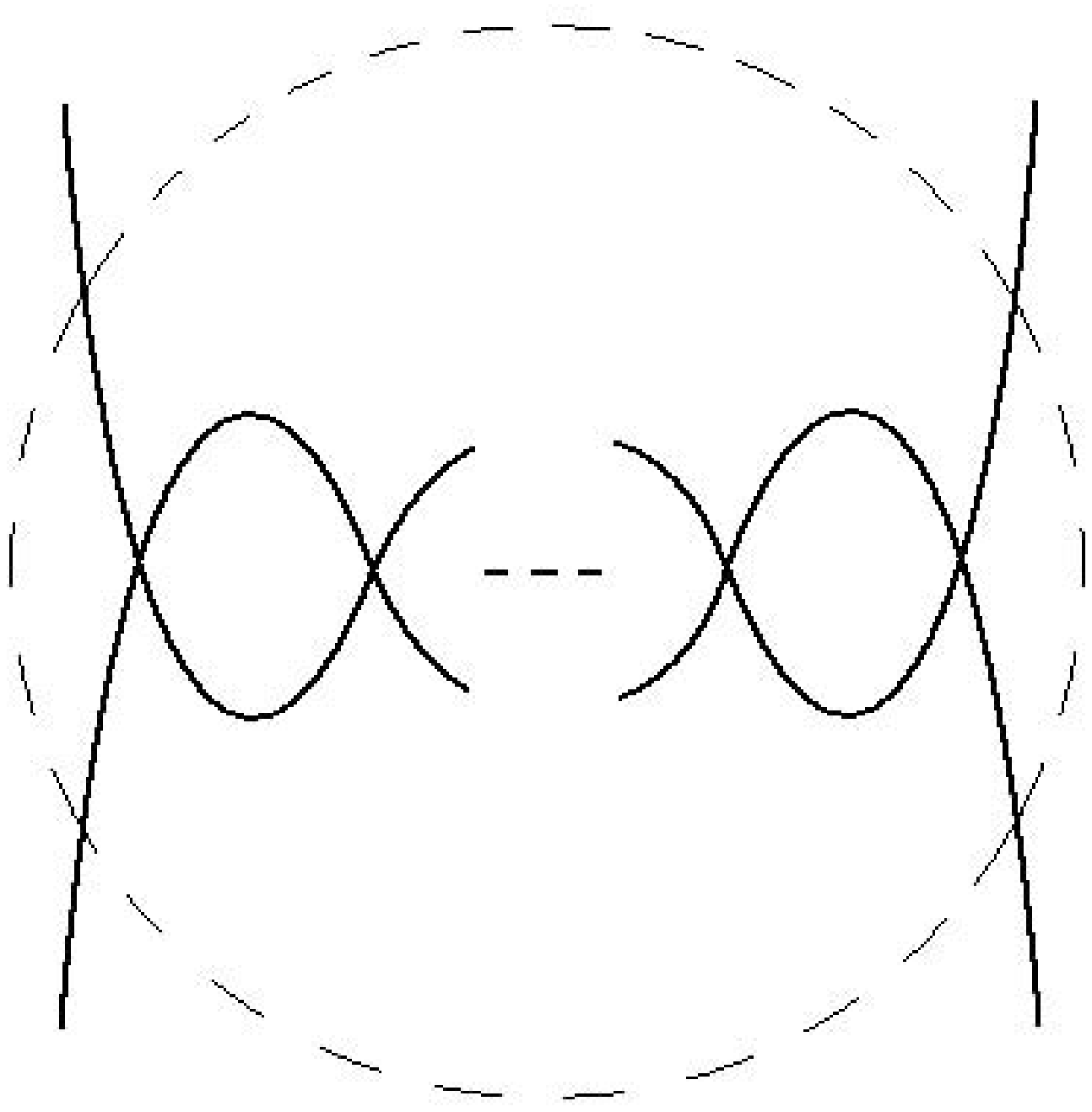} \qquad \includegraphics[scale=0.37]{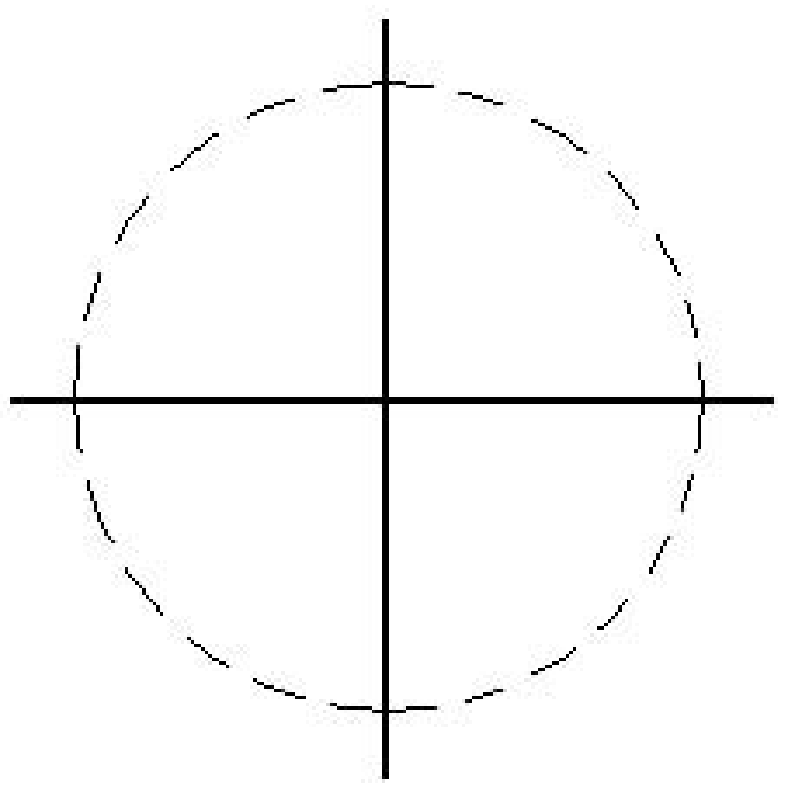} \\
\small (a) $(n,1)$-tangle\qquad\qquad\quad (b) $n$-tangle \qquad\qquad\qquad (c) $1$-tangle
    \end{center}
    \caption{Projections of various alternating tangles}\label{fig:non-alt-tangle}
  \end{figure}

Now we are ready to state our main theorems.

\begin{theorem}\label{theorem:n-1-nonalt}
Let $n\ge2$ and let $D$ be a prime $(n,1)$-nonalternating minimal crossing knot diagram having a
nonalternating triangular face $F_3$. Suppose that faces $F_1$, $F_2$, $F_3$, $F$, edges $e_1$, $e_2$ and a vertex $q$ of $D$ are labeled as in Figure~\ref{fig:(2,1)-nonalt}.
Then $\alpha(D)<c(D)$ if $D$ satisfies the two conditions below:
\begin{enumerate}
\item The face $F$ satisfies ${e_1 \cup e_2}\subset \partial F$ and $F\cap{(F_1 \cup F_2)}=\emptyset$.
\item There are two vertices $v\in{\partial F\cup\{q\}}$, $w\in{\partial F_2}$ and a string $a_{vw}$ of $D$ joining $v$ and $w$ such
that no edge of $\partial F_1 \cup \partial F_3$ is contained in $a_{vw}$.
\end{enumerate}
\end{theorem}

\begin{figure}[h]
         \begin{center}
              \includegraphics[scale=0.9]{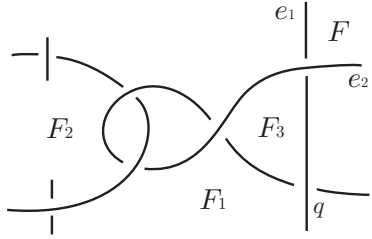}
        \end{center}
        \caption{{$(2,1)$-nonalternating diagram}}\label{fig:(2,1)-nonalt}
\end{figure}

\begin{theorem}\label{theorem:n-nonalt}
Let $n\ge2$ and let $D$ be a prime, $n$-nonalternating and minimal crossing knot diagram having a
nonalternating triangular face $F_3$. Suppose that faces $F_1$, $F_2$, $F_3$, $F$, edges $e_1$, $e_2$ and a vertex $q$ of $D$ are labeled as in Figure~\ref{fig:2-nonalt}.
Then $\alpha(D)<c(D)$ if $D$ satisfies the three conditions below:
\begin{enumerate}
     \item The face $F$ satisfies ${e_1 \cup e_2}\subset \partial F$ and $F\cap{(F_1 \cup F_2)}=\emptyset$.
     \item There are two vertices $v\in \partial F$, $w\in{\partial F_2\setminus \{q\}}$ and a string $a_{vw}$ of $D$ joining $v$ and $w$ such that no edge of $\partial F_1 \cup \partial F_2 \cup \partial F_3$ is contained in $a_{vw}$.
      \item $\partial{F_2}$ consists of at least $n+3$ edges.
\end{enumerate}
\end{theorem}

\begin{figure}[ht]
         \begin{center}
              \includegraphics[scale=0.8]{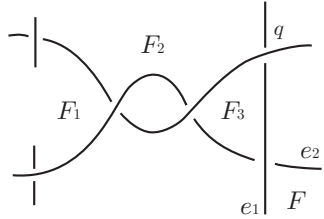}
        \end{center}
        \caption{{$2$-nonalternating diagram}}\label{fig:2-nonalt}
\end{figure}

\begin{theorem}\label{theorem:almost}
Let $D$ be a prime, almost alternating, and minimal crossing knot diagram having a nonalternating triangular face $F_3$. Suppose that faces $F_1$, $F_2$, $F_3$, an edge $e$ and a vertex $q$ of $D$ are labeled as in Figure~\ref{fig:almost}.
Let $F$ be the union of two faces containing $e$ in the intersection of their boundaries.
Then $\alpha(D)<c(D)$ if $D$ satisfies the two conditions below:
\begin{enumerate}
    \item $F\cap{(F_1 \cup F_2)}=\{q\}$
    \item There are two vertices $v\in{\partial F}$, $w\in{\partial F_2}$ and a string $a_{vw}$ of $D$ joining $v$ and $w$ such that no edge of $\partial F_1 \cup \partial F_3$ is contained in $a_{vw}$.
\end{enumerate}
\end{theorem}

\begin{figure}[t]
         \begin{center}
              \includegraphics[scale=0.66]{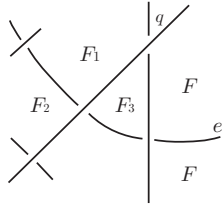}
        \end{center}
        \caption{{almost alternating}diagram}\label{fig:almost}
   \end{figure}

\section{Proofs of Main Theorems}

\subsection{Proof of Theorem~\ref{theorem:n-1-nonalt}}
We give a proof of the case $n=2$. It can be easily adapted for $n>2$.

Let $D^\prime$ be the diagram obtained from $D$ by a type $3$ Reidemeister move over the face $F_3$.
Some vertices, edges, and faces of $D^\prime$ are labeled as in Figure~\ref{fig:(2,1)-nonalt-rm3}.
The vertices $v_1,\ldots,v_7$ are all distinct except in the case that $\partial F_1$ of $D$ consists of only
four edges where $v_1=v_2$.
The two conditions of the theorem 
are modified to the following conditions on the diagram $D^\prime$:

\begin{enumerate}
\item[($1^\prime$)]
The face $F^\prime$ satisfies ${e_1^\prime \cup e_2^\prime}\subset \partial{ F^\prime}$ and $F^\prime
\cap{(F_1^\prime \cup F_2^\prime)}=\emptyset$
\item[($2^\prime$)]
There are two vertices $v\in{\partial F^\prime}$ and $w\in{\partial F_2^\prime}$ and a string $a_{vw}$ of $D$ joining $v$ and $w$ such that no edge of $\partial F_1^\prime \cup \partial F_3^\prime$ is contained in $a_{vw}$. The case $a_{vw}=\overline{v_7 v_5}$ is excluded.
\end{enumerate}

  \begin{figure}[ht]
         \begin{center}
              \includegraphics[scale=0.9]{_2,1_-nonalt.eps}
              \quad\raise35pt\hbox{$\Longrightarrow$}\quad
              \includegraphics[scale=0.8]{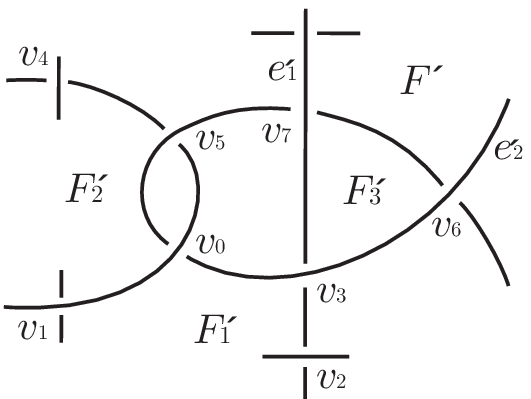}
        \end{center}
        \caption{{A type 3 Reidemeister move}}\label{fig:(2,1)-nonalt-rm3}
   \end{figure}

In this proof, we will construct a filtered tree whose successive closures gradually contain $\partial
F_1^\prime, \partial F_2^\prime$ and $\partial F_3^\prime$ without introducing bad edges and cutting
arcs.
During the construction, the nonalternating edges $\overline{v_2 v_3}$,
$\overline{v_4 v_5}$ and $\overline{v_6 v_7}$ will appear as doubly good edges in this order.

\medskip\noindent
{\sc Step 1.}
\emph{The edge $\overline{v_2v_3}$ becomes doubly good.}

Let $v_2^\prime$ be the vertex of $\partial F^\prime_1$ such that $v_0,v_3,v_2,v^\prime_2$ are adjacent along $\partial F^\prime_1$ in this order.
Applying Lemma~\ref{lem:d-good-1}, we obtain a good filtered tree $v_0=T_0\subset\cdots\subset T_m$ such that $T_m$ contains all vertices of $\partial F^\prime_1$ except $v_2$ and $v_3$, and its good extension $T_{m+1}$ along $\overline{v^\prime_2v_2}$. Extending once more along $\overline{v_0v_3}$, we obtain $T_{m+2}$. The edge $\overline{v_6v_7}$ obstructs the existence of a (B3) cutting arc for $T_{m+2}$. By Lemma~\ref{lem:d-good-1}, $\overline{v_2v_3}$ is doubly good in $\overline T_{m+2}$.

\medskip
\noindent{\sc Step 2.} \emph{The edge $\overline{v_4 v_5}$ becomes doubly good.}

Let $v_4^\prime$ be the vertex of $\partial F^\prime_2$ such that $v_0,v_5,v_4,v^\prime_4$ are adjacent along $\partial F^\prime_2$ in this order.
By Lemma~\ref{lem:step 2}, we have a sequence of good extensions  $T_{m+2}\subset\cdots\subset T_{n}$ such that $\overline T_{n}$ contains all edges of $\partial F^\prime_2$ except $\overline{v^\prime_4v_4}$, $\overline{v_4v_5}$ and $\overline{v_5v_0}$.
By ($2^\prime$) and by $D$ being prime and minimal,
the extension $T_n\cup\overline{v^\prime_4v_4}$ cannot be (B1) nor (B2). If it is (B3) then, applying Lemma~\ref{lem:d-good-1}, we can replace $T_n$ by a larger tree $T_{n^\prime}$ so that $T_{n^\prime}\cup\overline{v^\prime_4v_4}$ is a good extension. By the same reasons, the extension $T_{n^\prime+2}=T_{n^\prime}\cup\overline{v^\prime_4v_4}\cup\overline{v_0v_5}$ is not (B1) nor (B2).
The edge $\overline{v_3v_7}$ obstructs the existence of a (B3) cutting arc for $T_{n^\prime+2}$. By Lemma~\ref{lem:d-good-1}, $\overline{v_4v_5}$ is doubly good in $\overline T_{n^\prime+2}$.

\medskip
\noindent{\sc Step 3.} \emph{The edge $\overline{v_6 v_7}$ becomes doubly good.}

By ($1^\prime$) and by $D$ being prime and minimal,
the extension $T_{n^\prime+2}\cup\overline{v_3v_6}$ cannot be (B1) nor (B2).
If it is (B3) then, applying Lemma~\ref{lem:d-good-1}, we can replace $T_{n^\prime+2}$ by a larger tree $T_{n^{\prime\prime}}$ so that $T_{n^{\prime\prime}}\cup\overline{v_3v_6}$ is a good extension.
Now we consider the extension $T_{n^{\prime\prime}+2}=T_{n^{\prime\prime}}\cup\overline{v_3v_6}\cup\overline{v_3v_7}$.
It cannot be (B1) by one of the conditions ($1^\prime$), ($2^\prime$), $D^\prime$ being prime and minimal, depending on the location of the endpoint $c\in T_{n^{\prime\prime}}\cup\overline{v_3v_6}$ of the bad edge $e_1^\prime$.
It cannot be (B2) nor (B3) by one of the conditions ($1^\prime$), ($2^\prime$) and $D^\prime$ being prime, depending on the location of the endpoint $c\in T_{n^{\prime\prime}}\cup\overline{v_3v_6}$ of the cutting arc $\Gamma_p$ where $p=v_7$.
By Lemma~\ref{lem:d-good-2,3} and Corollary~\ref{cor:d-good-2,3}, the nonalternating edge $\overline{v_6v_7}$ a is doubly good edge of $\overline T_{n^{\prime\prime}+2}$.

\subsection{Proof of Theorem~\ref{theorem:n-nonalt}}
Let $D^\prime$ be the diagram obtained from $D$ by a type $3$ Reidemeister move over the face $F_3$.
Some vertices and faces of $D^\prime$ are labeled as in Figure~\ref{fig:2-nonalt-rm3}.
The vertices $v_0,\ldots,v_6$ are all distinct.
The three conditions of the theorem 
are modified to the following conditions on the diagram $D^\prime$:
\begin{enumerate}
\item[($1^\prime$)]
The face $F^\prime$ satisfies  ${e_1^\prime \cup e_2^\prime}\subset \partial{ F^\prime}$ and $F^\prime
\cap{(F_1^\prime \cup F_2^\prime)}=\emptyset$
\item[($2^\prime$)]
There are two vertices $v\in{\partial F^\prime}$ and $w\in{\partial F_2^\prime}$ and a string $a_{vw}$ of $D$ joining $v$ and $w$ such that no edge of $\partial F_1^\prime \cup \partial F_2^\prime \cup \partial F_3^\prime$ is contained in $a_{vw}$.
The case $a_{vw}=\overline{v_6 v_2}$ is excluded.
\item[($3^\prime$)]
 $\partial F_2^\prime $ consists at least $n+2$ edges.
\end{enumerate}

  \begin{figure}[h]
         \begin{center}
              \includegraphics[scale=0.8]{2-nonalt.eps}
              \quad\raise40pt\hbox{$\Longrightarrow$}\quad
              \includegraphics[scale=0.8]{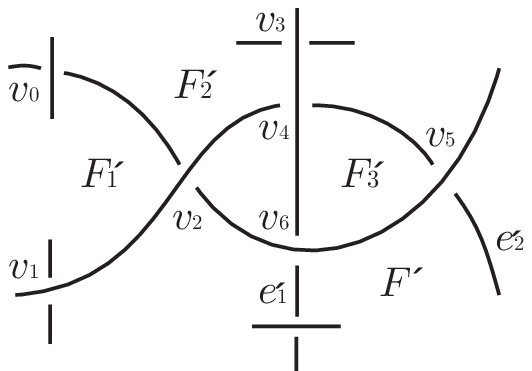}
        \end{center}
        \caption{{A type 3 Reidemeister move}}\label{fig:2-nonalt-rm3}
   \end{figure}

Similarly as in the proof of Theorem~\ref{theorem:n-1-nonalt}, we construct a filtered tree whose successive closures gradually contain $\partial F_1^\prime, \partial F_2^\prime$ and $\partial F_3^\prime$ without introducing bad edges and cutting
arcs.
During the construction, the nonalternating edges $\overline{v_1 v_2}$,
$\overline{v_3 v_4}$ and $\overline{v_5 v_6}$ will appear as doubly good edges.
We skip the detail.

\subsection{Proof of Theorem~\ref{theorem:almost}}
Let $D^\prime$ be the diagram obtained from $D$ by a type $3$ Reidemeister move over the face $F_3$.
Some vertices and faces of $D^\prime$ are labeled as in Figure~\ref{fig:almost-rm3}.
The vertices $v_0,\ldots,v_5$ are all distinct.
Let $F^\prime$ be the union of two faces containing $e^\prime$ in the intersection of their boundaries.
The two conditions of the theorem 
are modified to the following conditions on the diagram $D^\prime$:
\begin{enumerate}
\item[($1^\prime$)]
 $F^\prime \cap{(F_1^\prime \cup F_2^\prime)}=\{v_4\}$
\item[($2^\prime$)]
There are two vertices $v\in{\partial F^\prime}$ and $w\in{\partial F_2^\prime}$ and a string $a_{vw}$ of $D$ joining $v$ and $w$ such that no edge of $\partial F_1^\prime \cup \partial F_3^\prime$ is contained in $a_{vw}$.
The case $a_{vw}=\overline{v_4 v_3}$ is excluded.
\end{enumerate}
  \begin{figure}[h]
         \begin{center}
              \includegraphics[scale=0.8]{alm-alt-110317.eps}
              \quad\raise45pt\hbox{$\Longrightarrow$}\quad
              \includegraphics[scale=0.8]{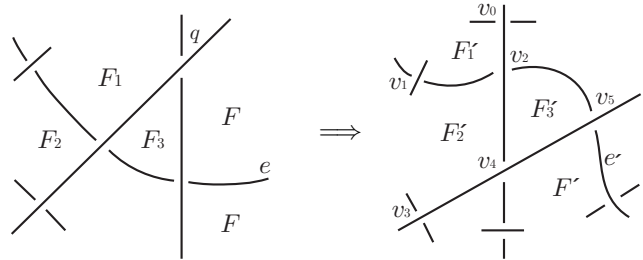}
        \end{center}
        \caption{{A type 3 Reidemeister move}}\label{fig:almost-rm3}
   \end{figure}

Similarly as in the proof of Theorem~\ref{theorem:n-1-nonalt}, we construct a filtered tree whose successive closures gradually contain $\partial F_1^\prime, \partial F_2^\prime$ and $\partial F_3^\prime$ without introducing bad edges and cutting
arcs.
During the construction, the nonalternating edges $\overline{v_1 v_2}$,
$\overline{v_3 v_4}$ and $\overline{v_4 v_5}$ will appear as doubly good edges.
We skip the detail.


\section{Examples and non-examples}\label{sec:examples}
\subsection{Examples of Theorem~\ref{theorem:n-1-nonalt}}
The first of the three diagrams in each figure is the minimal diagram which is $(n,1)$-nonalternating. The second is obtained by a type 3 Reidemeister move over the face $F_3$. The third is marked with a good filtered tree whose closure has three doubly good edges. If a black-thickend edge is $\overline{ij}$ with $i<j$ then it is the $j$-th edge of the tree. This comment also applies to the subsection \ref{subsec:n-nonalt-examples}.

\begin{figure}[h]
         \centering 
              \includegraphics[scale=0.78]{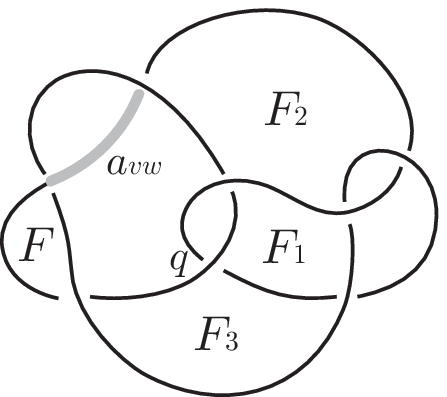}
              \includegraphics[scale=0.78]{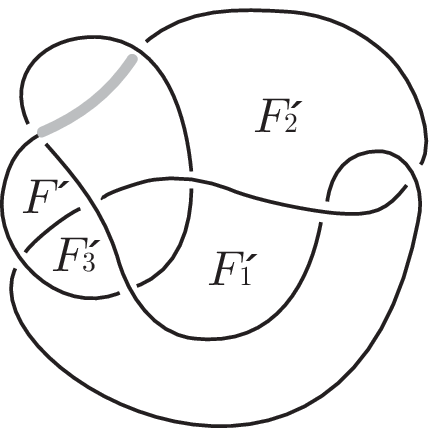}
              \includegraphics[scale=0.78]{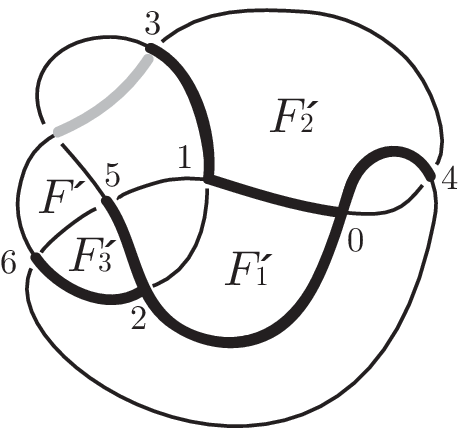}
         \caption{\small{$(2,1)$-nonlaternating : $\alpha(8n3)=7$}}
      \end{figure}

\begin{figure}[h]
         \begin{center}
              \includegraphics[scale=0.78]{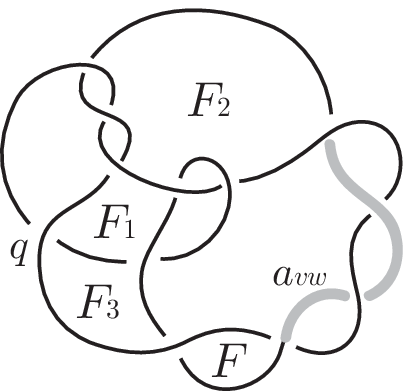}
              \includegraphics[scale=0.78]{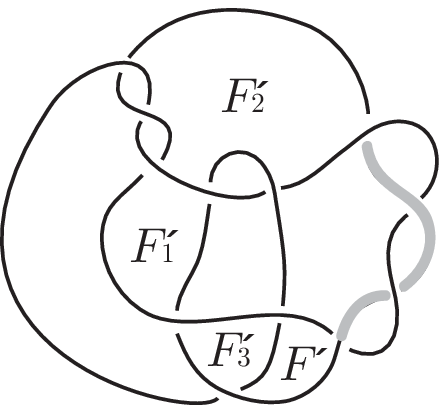}
              \includegraphics[scale=0.78]{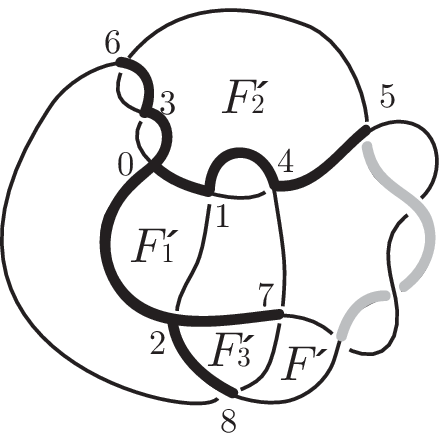}
          \end{center}
         \caption{{$(3,1)$-nonlaternating : $\alpha(12n475)=10$}}
      \end{figure}

\begin{figure}[h!]
         \begin{center}
              \includegraphics[scale=0.78]{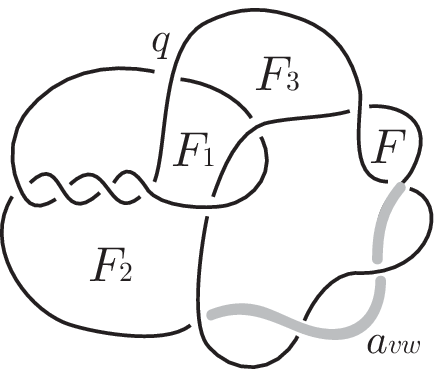}
              \includegraphics[scale=0.78]{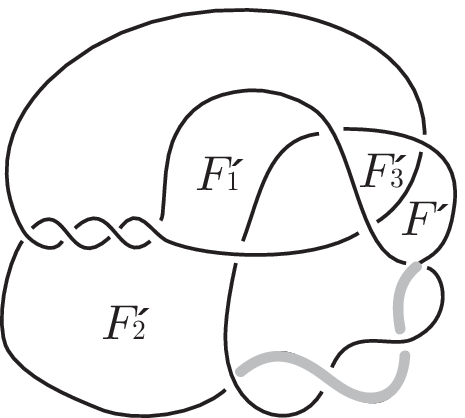}
              \includegraphics[scale=0.78]{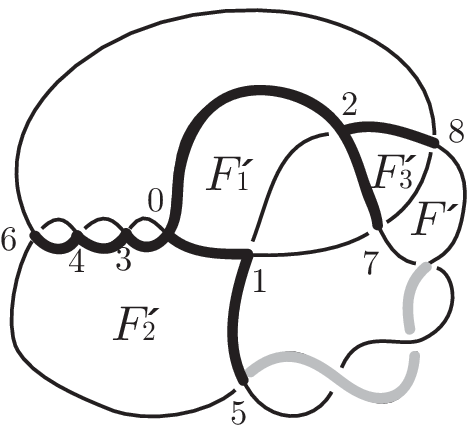}
          \end{center}
         \caption{{$(4,1)$-nonlaternating : $\alpha(12n725)=10$}}
      \end{figure}

\subsection{Non-examples of Theorem~\ref{theorem:n-1-nonalt}}
Each figure shows same diagram twice with different choices of faces $F_1$, $F_2$, $F_3$ and $F$. One can check that a condition of the theorem does not hold. This comment also applies to the subsection \ref{subsec:n-nonalt-nonexamples}.
\begin{figure}[h]
         \begin{center}
              \includegraphics[scale=0.8]{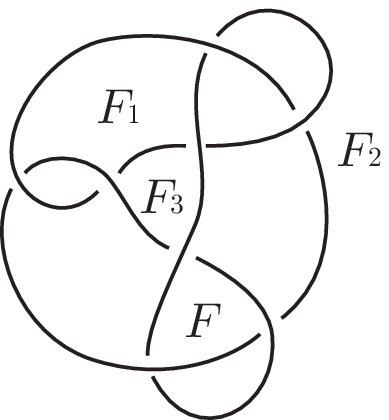} \hspace{1cm}
              \includegraphics[scale=0.8]{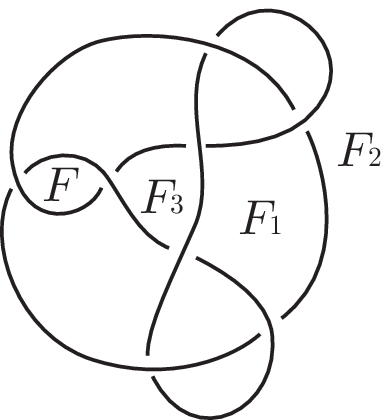}
          \end{center}
         \caption{{$(2,1)$-nonlaternating : $\alpha(8n2)=8$}}
      \end{figure}

\begin{figure}[h]
         \begin{center}
              \includegraphics[scale=0.8]{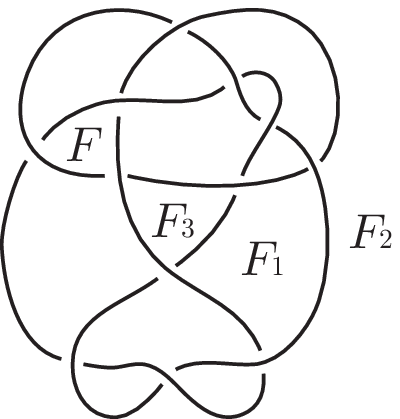} \hspace{1cm}
              \includegraphics[scale=0.8]{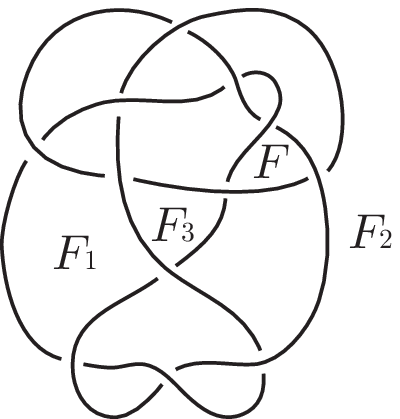}
          \end{center}
         \caption{{$(3,1)$-nonlaternating : $\alpha(12n699)=12$}}
      \end{figure}

\begin{figure}[h!]
         \begin{center}
              \includegraphics[scale=0.8]{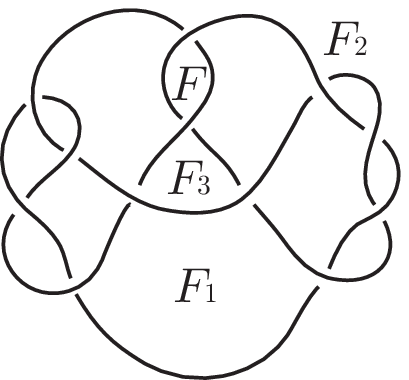} \hspace{1cm}
              \includegraphics[scale=0.8]{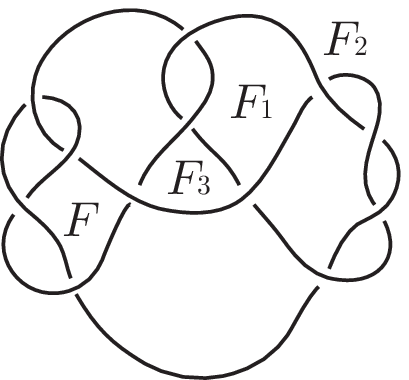}
          \end{center}
         \caption{{$(4,1)$-nonlaternating : $\alpha(12n305)=12$}}
      \end{figure}

\clearpage
\subsection{Examples of Theorem~\ref{theorem:n-nonalt}}\label{subsec:n-nonalt-examples}
{\ }
  \begin{figure}[h!]
         \begin{center}
             \includegraphics[scale=0.78]{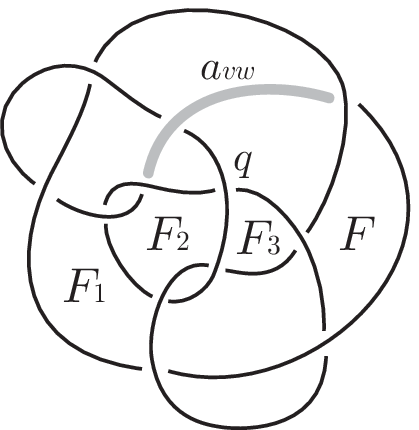}
             \includegraphics[scale=0.78]{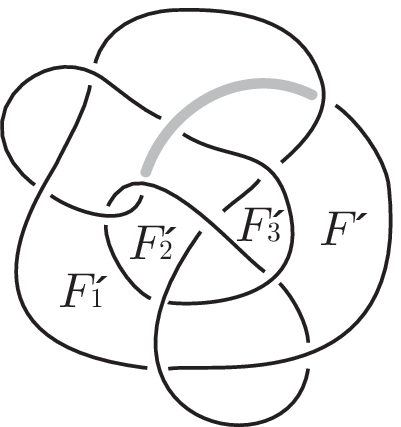}
             \includegraphics[scale=0.78]{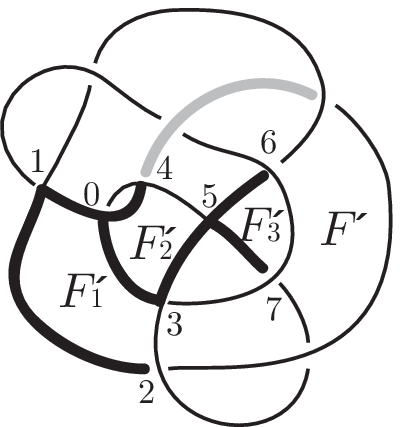}
        \end{center}
        \caption{{$2$-nonalternating : $\alpha(12n810)=11$}}
   \end{figure}

  \begin{figure}[h]
         \begin{center}
              \includegraphics[scale=0.78]{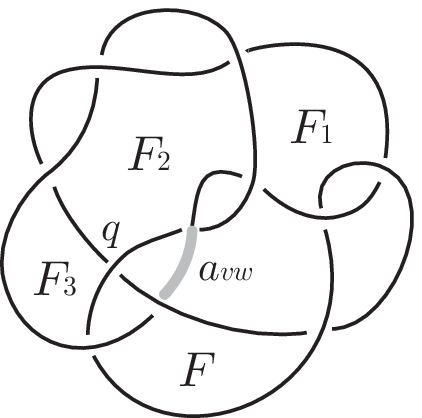}
              \includegraphics[scale=0.78]{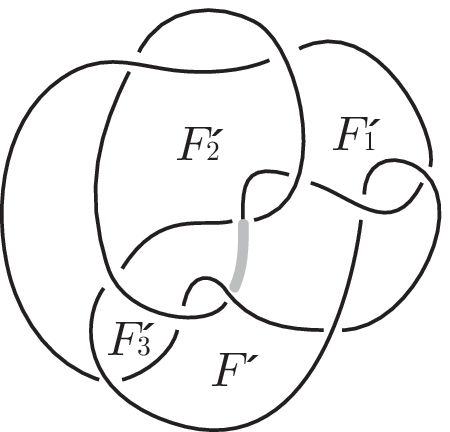}
              \includegraphics[scale=0.78]{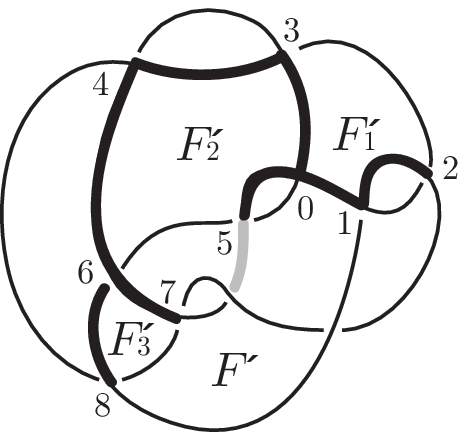}
        \end{center}
        \caption{{$3$-nonalternating : $\alpha(11n110)=10$}}
   \end{figure}

  \begin{figure}[h]
         \begin{center}
              \includegraphics[scale=0.78]{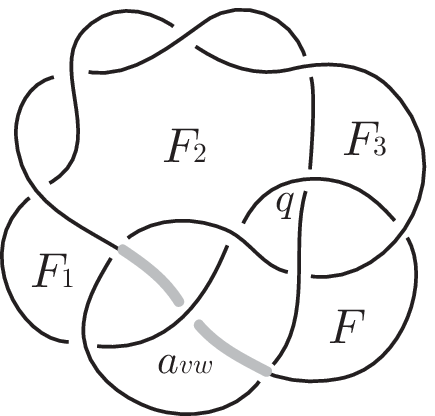}
              \includegraphics[scale=0.78]{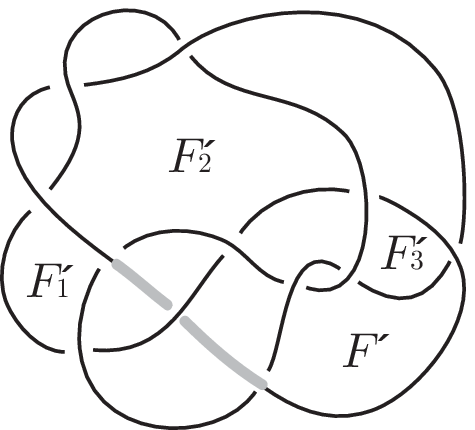}
              \includegraphics[scale=0.78]{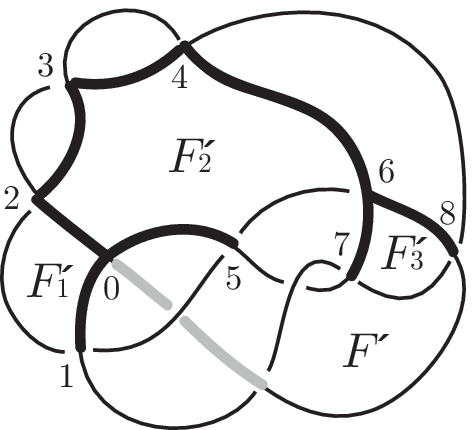}
        \end{center}
        \caption{{$4$-nonalternating : $\alpha(12n847)=11$}}
   \end{figure}

\subsection{Non-examples of Theorem~\ref{theorem:n-nonalt}}\label{subsec:n-nonalt-nonexamples}
  \begin{figure}[b]
         \begin{center}
             \includegraphics[scale=0.9]{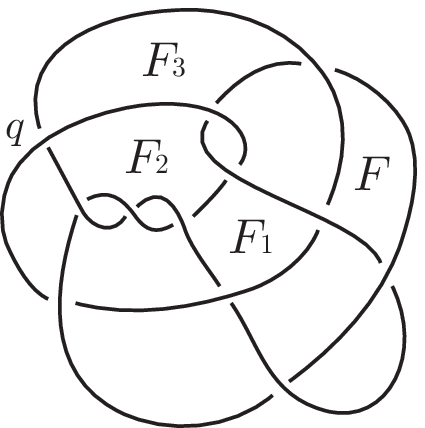} \hspace{1cm}
             \includegraphics[scale=0.9]{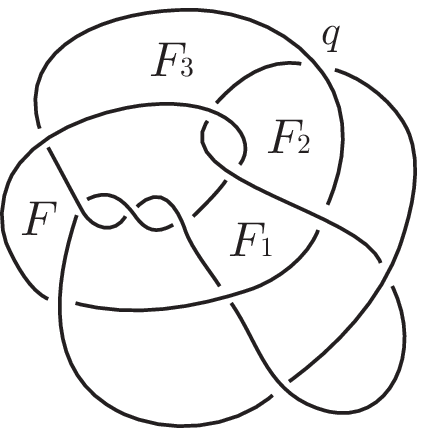}
        \end{center}
        \caption{{$2$-nonalternating : $\alpha(12n777)=12$}}
   \end{figure}

\clearpage
  \begin{figure}[t]
         \begin{center}
              \includegraphics[scale=0.9]{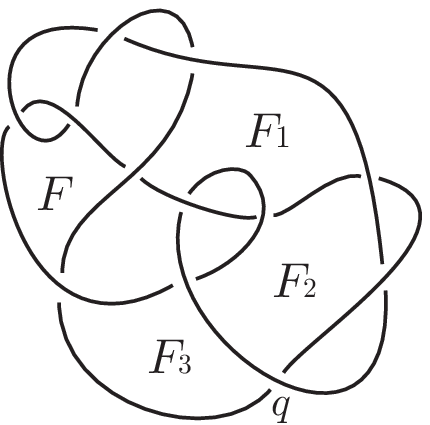}\hspace{1cm}
              \includegraphics[scale=0.9]{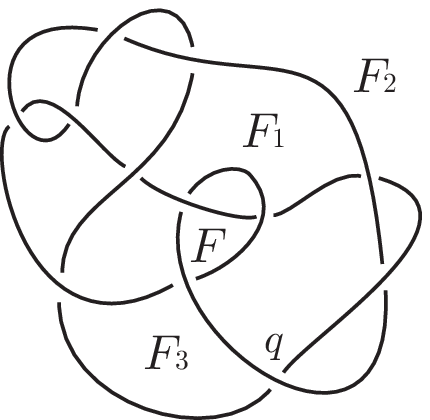}
        \end{center}
        \caption{{$3$-nonalternating : $\alpha(12n389)=12$}}
   \end{figure}

  \begin{figure}[h!]
         \begin{center}
              \includegraphics[scale=0.9]{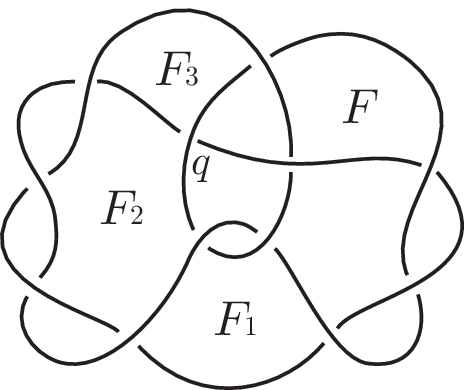} \hspace{1cm}
              \includegraphics[scale=0.9]{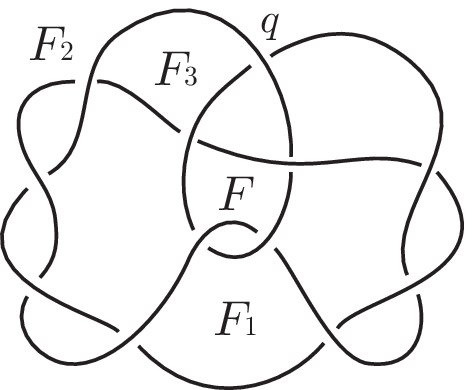}
        \end{center}
        \caption{{$4$-nonalternating : $\alpha(12n767)=12$}}
   \end{figure}


\section{Nonalternating knots with $\alpha(K)=c(K)-1$}
In \cite{N1999} Nutt identified all knots up to arc index 9.
In \cite{B2002} Beltrami determined arc index for prime knots up to 10 crossings.
In \cite{Jin2006} Jin et al. identified all prime knots up to arc index 10.
In \cite{Ng2006} Ng determined arc index for prime knots up to 11 crossings.
In \cite{Jin-Park2007} Jin and Park identified the prime knots up to arc index 11.

Using the Dowker-Thistlethwaite codes contained in Knotscape~\cite{knotscape}, we made lists of $13$ crossing knots  and $14$ crossing knots which are $(n,1)$-nonalternating, $n$-nonlaternating or almost alternating. Applying the conditions listed in the theorems, we were able to find the lists below.

\subsection{A partial list of 13 crossing knots with arc index 12}
The 13 crosssing knots in the lists below do not appear in the article \cite{Jin-Park2010} containing all prime knots up to arc index 11. Using the methods described in the proofs of main theorems, we were able to find grid diagrams of them with 12 vertical arcs.
In the grid diagrams below, we have the convention that the vertical edges pass over the horizontal edges.

\setlength{\unitlength}{1.05pt}

\subsection*{\boldmath$(2,1)$-nonalternating}\vrule width0pt depth10pt\newline
\noindent
{\small

}

\subsection*{\boldmath$(3,1)$-nonalternating}\vrule width0pt depth10pt\newline
{\small
%
}

\subsection*{\boldmath$2$-nonalternating}\hskip-8pt\footnote%
{Those marked with $^\star$ do not satisfy some conditions of Theorem~\ref{theorem:n-nonalt}, but satisfy $\alpha(D)<c(D)$.}
\vrule width0pt depth10pt\newline
{\small
 %
}

\subsection*{\boldmath$3$-nonalternating}\vrule width0pt depth10pt\newline
{\small
%
}

\subsection*{\boldmath$4$-nonalternating}\vrule width0pt depth10pt\newline
{\small
%
}

\subsection*{\boldmath$5$-nonalternating}\vrule width0pt depth10pt\newline
{\small
%
}

\subsection*{Almost alternating}\vrule width0pt depth10pt\newline
{\small
 %
 }

\subsection{A partial list of 14 crossing knots with arc index 13}
The 14 crosssing knots in the lists below have Kauffman $v$-spread equal to 11, hence their arc index is at least 13. Using the methods described in the proofs of main theorems, we were able to find grid diagrams of them with 13 vertical arcs.

\subsection*{\boldmath$(2,1)$-nonalternating}\vrule width0pt height15pt depth10pt\newline
%

\subsection*{\boldmath$(3,1)$-nonalternating}\vrule width0pt height15pt depth10pt\newline
%

\subsection*{\boldmath$(4,1)$-nonalternating}\vrule width0pt height15pt depth10pt\newline
%

\subsection*{\boldmath$(5,1)$-nonalternating}\vrule width0pt height15pt depth10pt\newline
%

\subsection*{\boldmath$2$-nonalternating}\vrule width0pt height15pt depth10pt\newline
%

\subsection*{\boldmath$3$-nonalternating}\vrule width0pt height15pt depth10pt\newline
%

\subsection*{\boldmath$4$-nonalternating}\vrule width0pt height15pt depth10pt\newline
%

\subsection*{\boldmath$6$-nonalternating}\vrule width0pt height15pt depth10pt\newline
%

\subsection*{Almost alternating}\vrule width0pt height15pt depth10pt\newline
%

\clearpage
\section*{Acknowledgments}
This research was supported by Basic Science Research Program
        through the National Research Foundation of Korea(NRF) funded
        by the Ministry of Education, Science and Technology(2010-0013742).

\end{document}